\documentclass[12pt]{article}
\usepackage{amsmath,amsfonts,amsthm,amssymb,mathrsfs,mathtools}
\usepackage{graphicx}
\usepackage[usenames,dvipsnames]{color}
\usepackage{float}
\usepackage{graphicx}
\usepackage{subfigure}
\usepackage{caption}
\usepackage[symbol]{footmisc}
\usepackage{cite}
\usepackage{url}
\usepackage{caption}
\usepackage{epstopdf}
\usepackage{hyperref}
\usepackage{color}

\newcommand{\dint}{\displaystyle\int}

\theoremstyle{plain}
\newtheorem{theorem}{Theorem}[section]
\newtheorem{hy}{Assumption}[section]
\newtheorem{corollary}[theorem]{Corollary}

\newtheorem{lemma}[theorem]{Lemma}
\newtheorem{proposition}[theorem]{Proposition}

\theoremstyle{definition}

\theoremstyle{remark}
\newtheorem{remark}[theorem]{Remark}

\numberwithin{equation}{section}
\numberwithin{theorem}{section}

\usepackage{a4wide}
\usepackage{geometry}
\begin{document}
	\renewcommand{\thefootnote}{\fnsymbol{footnote}}
	\begin{center}
		{\Large \textbf{Information-Based Approach: Pricing of a Credit Risky Asset in the Presence of Default Time}} \\[0pt]
		~\\[0pt] \textbf{Mohammed Louriki}\footnote[1]{Mathematics Department, Faculty of Sciences Semalalia, Cadi Ayyad University, Boulevard Prince Moulay Abdellah,	P. O. Box 2390, Marrakesh 40000, Morocco. E-mail: \texttt{louriki.1992@gmail.com}}
		\\[0pt]
	\end{center}
	\begin{abstract}
		We extend the  information-based asset-pricing framework by Brody, Hughston \& Macrina to incorporate a stochastic bankruptcy time for the writer of the asset. Our model introduces a non-defaultable cash flow $Z_T$ to be made at time $T$, alongside the time  $\tau$ of a possible bankruptcy of the writer of the asset are in line with the filtration generated by a Brownian random bridge with length $\nu=\tau \wedge T$ and pinning point $\sigma Z_T$, where $\sigma$ is a constant. Quantities $Z_T$ and $\tau$ are not necessarily independent. The model does not depend crucially on the interpretation of $\tau$ as a bankruptcy time. 
		We derived the price  process of the asset and compute the prices of associated options. The dynamics of the price process satisfy a diffusion equation. Employing the approach of P.-A.~ Meyer, we provide the explicit computation of the compensator of $\nu$. Leveraging  special properties of the bridge process, we also provide the explicit expression of the compensator of $Z_T\,\mathbb{I}_{[\nu,+\infty)}$. The resulting conclusion highlights the totally inaccessible property of the stopping time $\nu$. This characteristic is particularly suitable for financial markets where the time of default of a writer cannot be predictable from any other signal in the system until default happens.
	\end{abstract}
	\section{Introduction }
	\,\,\,\,\,\,\, Financial crises in the recent history, notably the 2008 collapse of Lehman Brothers, a culmination of the sub-prime mortgage crisis, have emphasized the critical need for modelling of credit risk. A particular role in this modelling is played by the pioneering work by Brody, Hughston and Macrina \cite{BHM2007}, which introduced an information-based approach for credit risk. This paper is based on the information-based asset pricing framework establishment in \cite{BHM2007} by Brody, Hughston and Macrina (hereafter referred to as BHM framework). We delve into the BHM framework, specifically examining upon essential components relevant to modelling defaultable bonds within a more generalized context. The goal is to enhance the applicability and robustness of this information-based approach to credit risk modelling. 	
	In \cite{BHM2007}, Brody, Hughston and Macrina construct a model that serves as the starting point of a new class of models designed to address certain limitations of the two primary approaches to credit risk: the structural approach and the family of reduced-form models:
	
	Structural models concentrate on a firm’s assets and liabilities, providing a mechanism for default. The probability of default is endogenous, typically triggered when the value  of the firm’s assets hits a barrier representing default. For instance, in the Merton’s model \cite{Merton}, the value process $V=\left(V_{t}, t \geq 0\right)$ of a firm is assumed to follow the dynamics of geometric Brownian motion.
	$$
	\mathrm{d} V_{t}=V_{t}\left(\mu \mathrm{d} t+\sigma \mathrm{d} B_{t}\right), S_{0}=x>L>0 .
	$$
	where $B=\left(B_{t}, t \geq 0\right)$ is a Brownian motion. The bankruptcy time $\tau$ of the company is defined as
	$$
	\tau:=\inf \left\{t>0: V_{t} \leq L\right\}.
	$$
	
	In contrast to structural models, under reduced-form models, default is no longer tied to the firm’s assets dropping below a specific threshold level.  Instead, default takes place based on an exogenous hazard rate process. 
	
	The BHM approach, developed by Brody et al.~ \cite{BHM2007}, aims to integrate elements from both structural and reduced-form models, effectively mitigating their respective limitations. This approach capitalizes on the strengths of each model: the economic and intuitive appeal of the structural approach, and the tractability and empirical fit of the reduced-form approach. 
	A key conceptual tool within this framework is the use of market information processes, implemented through  Brownian bridges with random pinning point. In their methodology, they model the flow of information accessible to market participants concerning a cash flow corresponding to a future random payment $Z_T$ at maturity $T$ with the natural, completed filtration generated by the process $\xi^T=(\xi^T_t, t\leq T)$, defined by
	\begin{equation}
		\xi_t^T=\sigma t Z_T+\beta_t^T.
	\end{equation}
	where $\beta^T = (\beta_t^T, 0\leq t \leq T)$ is a Brownian bridge on the deterministic time interval $[0, T]$, going from zero to zero. The process $\beta^T$ is independent of $Z_T$, and $\sigma > 0$ is a positive constant. The concept idea is that the true information, represented by the term $\sigma t Z_T$, about the cash flow is disturbed by some noisy information (rumors, mispricing, etc.), represented by the bridge process. This makes that at time $t = 0$ there is no information, and only at time $t = T$, the precise information about the true value of $Z_{T}$ is revealed (investors have perfect information about the values of $Z_{T}$ at time $T$). 
	In this context, it is worth noting the works by Brody, Hughston and Macrina \cite{BHM2010}, \cite{BHM2008} and \cite{BHM}, Rutkowski and Yu \cite{RuYu}, Hoyle, Hughston and Macrina \cite{HHM}, Brody and Law \cite{BL}, Hoyle, Macrina and Meng\"ut\"urk \cite{HMM}, and  Meng\"ut\"urk \cite{Men}. In this credit risk models, information concerning the random cash-flow $Z_T$ is modeled explicitly, but the default time is not, as the payment is contractually set to take place at maturity only. However, information about definite default before maturity will immediately influence the price. Motivated by the challenge of modelling information related to the default time of a financial company, this paper proposes a mathematical model. In this model, a non-defaultable cash flow with an agreed single payment $Z_T$ at maturity $T$ and the time of bankruptcy $\tau$ of the writer of the associated asset are both modelled with the filtration generated by the information process $\zeta^T=(\zeta^T_t, t\leq T)$, defined by
	\begin{equation}
		\zeta^T_t=W_{t\wedge \tau}-\dfrac{t\wedge \tau}{\tau\wedge T} W_{\tau\wedge T}+\sigma \dfrac{t\wedge \tau }{\tau\wedge T} Z_T.
	\end{equation}
	The value of $Z_T$ becomes known at the positive random time $\tau\wedge T$. Specifically, the cash flow is measurable at time $T$ when the contract expires. However, if the counterpart declares bankruptcy at time $\tau$ when $\tau<T$, the exact value of $Z_T$ becomes ascertainable. Since our primary interest lies in determining whether default will occur before the maturity date or not, we consider random time $\nu=\tau\wedge T$ taking values in $(0,T]$ with the understanding that if $\nu<T$ then the default time is occurred at $\nu$ before the maturity date $T$, while if $\nu=T$ the default time has not occurred before the maturity date $T$. 
	The completed natural filtration $\mathbb{F}^{\zeta^T}$ generated by $\zeta^T$ offers partial information about the cash flow $Z_T$ and the bankruptcy time $\tau$ before it occurs. The concept of using the length of a bridge process for modelling the bankruptcy time dates to the work of Bedini et al.~ \cite{BBE}, Erraoui et al.~ \cite{EHL}, \cite{EHL(Levy)}, \cite{EHL(Brownian-Levy)} and  \cite{EL}. We mainly deal with the case where $Z_T$ takes values in $\{z_1,\ldots, z_l\}$ with a priori probability $\mathbb{P}(Z_T=z_i)=p_i$. Additionally, the conditional distribution of $\tau$ given $Z_T=z$ admits a continuous density function $f(.,z)$ with respect to the Lebesgue measure on $\mathbb{R}_+$. 
	Once our model is established, leveraging the properties of our market information process $\zeta^T$ such as the Markov property and the canonical decomposition in its natural filtration, we derive an expression for the value process of a contract that delivers the cash flow $Z_T$ at time $T$, we provide estimates of the a priori unknown bankruptcy time, relying on observations of the information process $\zeta^T$ up to time $t$. Subsequently, we analyse the dynamics of the price process and demonstrate that the price process satisfies a diffusion equation. The well-known Doob-Meyer decomposition theorem states that if $X$ is a sub-martingale, then it can be expressed in the form $X=M+K$, where $M$ is a local martingale, and $K$ is a c\`ad-l\`ag increasing predictably process with $K_0=0$. A special case of interest in the theory of credit risk is the case where the process given by
	\begin{equation}
		\mathbb{I}_{[S, +\infty)}-K \label{eqcompensatorintro}
	\end{equation} 
	is a martingale. The process $K$ in \eqref{eqcompensatorintro} is known as the compensator of the stopping time $S$, and the relation between $S$ and its compensator $K$ can be characterized as follows: $S$ is predictable if and only if $K$ is a.s.~ constant except for a possible unit jump, $S$ is accessible if and only if $K$ is a.s.~ purely discontinuous, while $S$ is totally inaccessible if and only if $K$ is a.s.~ continuous (see \cite{Kall} , Chapter 25]). 
	It is evident that the bankruptcy time $\tau$ holds a central focus for market  agents, who seek to acquire comprehensive information about it. To gather pertinent insight on bankruptcy time $\tau$, it becomes crucial to determine whether the nature of the stopping time $\nu=\tau\wedge T$ is predictable, accessible, or totally inaccessible. Market agents can anticipate the occurrence of bankruptcy when $\nu$ is predictable, whereas the bankruptcy takes place unexpectedly  when $\nu$ is totally inaccessible. To this end, following the P.-A.~ Meyer approach, we provide the explicit computation of the compensator of $\nu$ and we show that it is given by
	\begin{equation}
		K_t=\displaystyle\sum\limits_{k=1}^l\displaystyle\int_{0}^{t\wedge \nu}\dfrac{\dfrac{p_k}{\sqrt{2\pi s}}\exp\Big(\frac{\sigma^2z_k^2}{2s}\Big)f(s,z_k)}{\displaystyle\sum\limits_{i=1}^l\bigg[\displaystyle\int_{s}^{T}\Phi_{s,r}(z_k,z_i)f(r,z_i)\mathrm{d}r+F(T,z_i)\,\Phi_{s,T}(z_k,z_i)\bigg]p_i}\mathrm{d}L^{\zeta^T}(s,\sigma z_k),
	\end{equation}
	where
	\begin{equation}
		F(T,z)=\int_{T}^{+\infty}f(r,z)\mathrm{d}r,
	\end{equation}
	\begin{equation}
		\Phi_{s,t}(x,y)=\sqrt{\frac{t}{t-s}}\exp\Big(-\frac{\sigma^2}{2}\Big(\frac{(y-x)^2}{t-s}+\frac{y^2}{t}\Big)\Big),\,s<t,\,y\in \mathbb{R}
	\end{equation}
	and $L^{\zeta^T}(t,x)$ is the local time of the information process $\zeta^T$ at level $x$ up to time $t$. Utilizing the characterization of the special categories of stopping times based on the regularity of their associated compensators, we establish that $\nu=\tau\wedge T$ is a totally inaccessible stopping time since its compensator $K$ is continuous. This characteristic is particularly fitting for financial markets where the timing of a writer default cannot be predictable from any other signal in the system until default event occurs. Leveraging the fact that the cash flow $Z_T$ is $\mathcal{F}_{\nu}^{\zeta^T}$-measurable, we further derive expression of the compensator associated with the process $Z_T\,\mathbb{I}_{[\nu, +\infty)}$. 
	For an in-depth exploration of compensator concepts, we recommend chapter 3 in Protter \cite{P}, which extensively covers the study of compensators. Additional information and examples can be found in Aksamit and Jeanblanc \cite{AJ}. Also, refer to Janson et al.~ \cite{A} and \cite{JMP}  for conditions ensuring that the compensator is absolutely continuous with respect to the Lebesgue measure. In \cite{BBEcompensator}, Bedini et al.~ tackle the problem of explicitly computing the compensator of $\tau$ within the filtration generated by a Brownian bridge over the random interval $[0, \tau]$, going from zero to zero.
	
	The paper is organized as follows. 
	In Section 2, we introduce our model and examine the fundamental properties of our market information process. 
	Section 3 is dedicated to specify the pricing formula for a default bond. We provide the dynamics of the price process and demonstrate that the price process satisfies a diffusion equation. 
	In section 4, we focus on computing the the compensator of the random time $\nu$, which is the compensator of the process $H=(\mathbb{I}_{\{\tau\leq t\}}, t\geq 0)$. Additionally, leveraging the special properties of the market information process, we derive the explicit expression of the compensator of the process $Z_T\,H$.
	
	The following notations will be used throughout the paper: For a complete probability space $(\Omega,\mathcal{F},\mathbb{P})$, $\mathcal{N}_{\mathbb{P}}$ denotes the collection of $\mathbb{P}$-null sets. We denote by $\delta_x$ the Dirac measure centered on some fixed point $x$. If $\theta$ is a random variable, then we denote by $\mathbb{P}_{\theta}$ the law of $\theta$ under $\mathbb{P}$. If $E$ is a topological space, then the Borel $\sigma$-algebra over $E$ will be denoted by $\mathcal{B}(E)$. The characteristic function of a set $A$ is written $\mathbb{I}_{A}$. The function $p(t, x, y)$, $x, y\in \mathbb{R}$, $t\in\mathbb{R}_+$, denotes the Gaussian density function with variance $t$ and mean $y$, if $y=0$, for simplicity of notation we write $p(t,x)$ rather than $p(t, x, 0)$. Finally for any process $X=(X_t,\, t\geq 0)$ on $(\Omega,\mathcal{F},\mathbb{P})$, the completed natural filtration of $X$ will be denoted by $\mathbb{F}^{X}=(\mathcal{F}^{X}_t)_{t\geq 0}$: $$\mathcal{F}^{X}_t:=\sigma(X_s,\,\,s\leq t)\vee \mathcal{N}_{\mathbb{P}}.$$
	
	\section{Market information process}
	\label{sect2}
	\quad\,\,In this section, we explore an extension of the information-based approach for credit risk introduced by Brody, Hughston and Macrina \cite{BHM2007}. While their model explicitly models the information about the random cash-flow $Z_{T}$ the default time is fixed to $T$. Motivated by the challenge of modeling the information regarding the default time we propose an information-based model in which the cash flow and the default time are explicitly modeled. Let $\tau$ be a strictly positive random time representing the bankruptcy time. We consider the market information process $\zeta^T=(\zeta^T_t, t\leq T)$, defined by
	\begin{equation}
		\zeta^T_t=W_{t\wedge \nu}-\dfrac{t\wedge \nu}{\nu} W_{\nu}+\sigma \dfrac{t\wedge \nu}{\nu} Z_T,\,\, t\in [0, T],
	\end{equation}
	where $\nu=\tau\wedge T$. The value of $Z_T$ becomes known at the positive random time $\tau$. Specifically, the cash flow is measurable at time $T$ when the contract expires. However, if the counterpart declares bankruptcy at time $\tau$ when $\tau<T$, people can ascertain the exact value of $Z_T$. In this context, the completed natural filtration $\mathbb{F}^{\zeta^T}$ generated by $\zeta^T$ provides partial information about the cash flow $Z_T$ and the bankruptcy time $\tau$. Similar to \cite{BHM2007}, our model is determined through the following postulates:
	\begin{enumerate}
		\item[(i)] A random variable $Z_T$, a Brownian motion $W$, and a strictly positive random time $\tau$ are given on some probability space $(\Omega, \mathbb{F}, \mathbb{P})$ such that the Brownian motion $W$ is independent of $(\tau,Z_T)$.
		\item[(ii)] The filtration $\mathbb{F} = \mathbb{F}^{\zeta^T} = (\mathcal{F}_t^{\zeta^T})_{0\leq t \leq T}$ is generated by the market information process $\zeta^T$, given by the formula
		\begin{equation}
			\zeta^T_t=W_{t\wedge \nu}-\dfrac{t\wedge \nu}{\nu} W_{\nu}+\sigma \dfrac{t\wedge \nu}{\nu} Z_T,\,\, t\in [0, T],
		\end{equation}
		where $\nu=\tau \wedge T$.
		\item[(iii)] The short-term interest rate $r$ is deterministic. Hence, the value $P_t^T$ at time $t$ of a default-free bond maturing at time $T$ equals
		\begin{equation}
			P_t^T=\exp\bigg(-\int_t^Tr(u)\mathrm{d}u\bigg),\quad t\in [0,T].
		\end{equation}
		\item[(iv)] The probability measure $\mathbb{P}$ is the risk-neutral probability, in the sense, that the time-$t$ price of a cash flow $Z_T$, due at time $T$, is given by an expression of the form
		\begin{equation}
			B_t^T=P_t^T\, \mathbb{E}[Z_T\vert \mathcal{F}_t^{\zeta^T}].
		\end{equation}
	\end{enumerate}
	In this study we deal with the case where $Z_T$ takes the values $\{z_1,\ldots,z_l\}$, $l\in \mathbb{N}$, with a probability $\mathbb{P}\left(Z_T=z_i\right)=p_i$. It is natural to make the assumption that the bankruptcy satisfies the condition that $\mathbb{P}(\tau\leq T)<1$, indicating that the default event may occur prior to the maturity date $T$. Moreover, our main assumption is the following:
	\begin{hy}
		The conditional distribution of $\tau$ given $Z_T=z$, denoted as $\mathbb{P}_{\tau\vert Z_T=z}$, possesses a continuous density function $f(.,z)$ with respect to the Lebesgue measure on $\mathbb{R}_+$.
	\end{hy}
	In the context of our assumption, wherein the conditional distribution of $\tau$ given $Z_T=z$ has a continuous density function $f(.,z)$, we explore various examples. Let $B$ as a Brownian motion independent of $W$ and $Z_T$, and let $\mu$ as a strictly positive real number.
	
	1. Assume that $z_i \neq 0$ for all $i \in \{1,\ldots,l\}$. Let $\tau_{Z_T}=\inf\{t>0:\,B_t=Z_T\}$, the first hitting time of $Z_T$ for the Brownian motion $B$. In this case, the conditional probability can be expressed as
	\begin{equation}
		\mathbb{P}(\tau_{Z_T}\in \mathrm{d}t\vert Z_T=z_i)=\frac{\vert z_i\vert}{\sqrt{2 \pi} t^{3 / 2}} \exp \left(-\frac{z_i^2}{2 t}\right)\mathbb{I}_{\{t>0\}}\mathrm{d}t.
	\end{equation}
	
	2. Assume $\min\limits_{1\leq i \leq l}(z_i)>0$. Let $\tau^{\mu}_{Z_T}=\inf\{t>0:\,\,B^{\mu}_t=Z_T\}$ be the first hitting time of $Z_T$ for, the Brownian motion with drift $\mu$, $B^{\mu}=(B_t+\mu t, t\geq 0)$. In this case, the conditional probability is given by
	\begin{equation}
		\mathbb{P}(\tau^{\mu}_{Z_T}\in \mathrm{d}t\vert Z_T=z_i)=\frac{ z_i}{\sqrt{2 \pi} t^{3 / 2}} \exp \left(-\frac{(z_i-\mu t)^2}{2 t}\right)\mathbb{I}_{\{t>0\}}\mathrm{d}t.
	\end{equation}
	
	3. Assuming $\min\limits_{1\leq i \leq l}(z_i)>0$, let $\sigma^{\mu}_{Z_T}=\sup\{t>0:\,\,B^{\mu}_t=Z_T\}$ denote the last time that $B^{\mu}$ hits $Z_T$. The conditional probability in this case is given by
	\begin{equation}
		\mathbb{P}(\sigma^{\mu}_{Z_T}\in \mathrm{d}t\vert Z_T=z_i)=\dfrac{\mu}{\sqrt{2\pi t}}\exp\bigg(-\dfrac{1}{2t}(\mu t-z_i)^2\bigg)\mathbb{I}_{\{t>0\}}\mathrm{d}t.
	\end{equation}
	
	4.  Assuming $\min\limits_{1\leq i \leq l}(z_i)>0$, let $R$ be a Bessel process of dimension $\delta > 2$ with index $\nu$, and let $\Lambda^{\delta}_{Z_T}$ be the last time that $R$ hits $Z_T$. The corresponding conditional probability, given that $R$ is independent of $Z_T$, is expressed as
	\begin{equation}
		\mathbb{P}(\Lambda^{\delta}_{Z_T}\in \mathrm{d}t\vert Z_T=z_i)=\frac{1}{t \Gamma(\nu)}\left(\frac{z_{i}^2}{2 t}\right)^\nu \exp\bigg( \dfrac{-z_i^2}{2t}\bigg)\mathbb{I}_{\{t>0\}}\mathrm{d}t
	\end{equation}
	where, $\Gamma$ is the gamma function.
	
	There are various examples to consider. For instance, we consider $\tau$ as either the hitting time or last passage time of a diffusion process or a L\'evy process. See, e.g., \cite{AK}, \cite{APP}, \cite{BS} and \cite{JYC}.
	
	The bridge with random length and pinning point has been investigated by Louriki \cite{L} in particularly in the scenario where the pinning point and the bridge length are independent. In this current paper, we model our market information process $\zeta^T$ as Brownian bridge with length $\nu=\tau\wedge T$ and pinning point $\sigma Z_T$. Notably, the variables $\tau$ and $Z_T$ are not necessarily independent. To gather crucial insights into the time of bankruptcy $\tau$ and the cash flow $Z_T$, a thorough examination of the basic properties of the market information $\zeta^T$ becomes essential. Hence, we embark on the task of reassessing these properties, scrutinizing the Brownian bridge without assuming independence between the length and the pinning point. We begin with the central point underlying all subsequent findings:
	\begin{proposition} For all $0<t<T$, we have, $\mathbb{P}$-a.s.,
		\begin{equation}
			\mathbb{I}_{\{ \nu\leq t \}}=\displaystyle\sum\limits_{i=1}^l\mathbb{I}_{\{\zeta^T_t=\sigma z_i\}}.\label{eqmodification}
		\end{equation}
	\end{proposition}
	\begin{proof}
		Set $\Delta=\{\sigma z_1,\ldots,\sigma z_l\}$ and for $(r,z)\in (0,+\infty)\times\mathbb{R}$ let $W^{r,z}$ be the Brownian bridge with length $r$ and pinning point $\sigma z$ associated with $W$:
		\begin{equation}
			W^{r,z}_t=W_{t\wedge r}-\dfrac{t\wedge r}{r} W_{r}+\sigma \dfrac{t\wedge r}{r} z.
		\end{equation}
		Exploiting the fact that $\mathbb{P}(Z_T\notin \{z_1,\ldots,z_l\})=0$, the set $\Delta$ is of Lebesgue measure zero, and that the law of $W_{t}^{r,z}$ is absolutely continuous with respect to the Lebesgue measure when $0<t<r$, we derive
		\begin{align*}
			\mathbb{P}(\{\zeta^T_t \in \Delta\} &\bigtriangleup \{\nu\leq t\})=\mathbb{P}(\zeta^T_t\in \Delta,t<\nu)+\mathbb{P}(\zeta_t\in \Delta^c,\nu\leq t)\\
			&=\dint_{\mathbb{R}}\dint_{t}^{+\infty} \mathbb{P}(\zeta^T_t \in \Delta \vert \nu=r,Z_T=z)\mathbb{P}_{(Z_T,\nu)}(\mathrm{d}r,\mathrm{d}z)\\&+\dint_{\mathbb{R}}\dint_{0}^{t} \mathbb{P}(\zeta^T_t\in \Delta^c \vert \nu=r,Z_T=z)\mathbb{P}_{(Z_T,\nu)}(\mathrm{d}r,\mathrm{d}z)\\
			&=\dint_{\mathbb{R}}\dint_{t}^{+\infty} \mathbb{P}(W_{t}^{r,z} \in \Delta)\mathbb{P}_{(Z_T,\nu)}(\mathrm{d}r,\mathrm{d}z)+\dint_{\mathbb{R}}\dint_{0}^{t} \mathbb{P}(W_{t}^{r,z}\in \Delta^c )\mathbb{P}_{(Z_T,\nu)}(\mathrm{d}r,\mathrm{d}z)
			\\  
			&=\dint_{\mathbb{R}}\dint_{t}^{+\infty} \mathbb{P}(W_{t}^{r,z}\in \Delta)\mathbb{P}_{(\nu,Z)}(\mathrm{d}r,\mathrm{d}z)+ \,\mathbb{P}(\nu\leq t,\sigma Z_T\in \Delta^c)\\
			&=0.
		\end{align*}
		This concludes the proof of \eqref{eqmodification}.
	\end{proof}
	\begin{remark}
		The random time $\nu$ is a stopping time with respect to $\mathbb{F}^{\zeta^T}$, the completed filtration generated by $\zeta^T$.
	\end{remark}
	\begin{theorem}
		The market information process $\zeta^T$ is a Markov process with respect to $\mathbb{F}^{\zeta^T}$, i.e., 
		\begin{equation}
			\mathbb{E}[g(\zeta^T_u)\vert\mathcal{F}^{\zeta^T}_t]=\mathbb{E}[g(\zeta^T_u)\vert\zeta^T_t],\,\,\mathbb{P}\text{-a.s.},
		\end{equation}
		for all $0 \leq t < u\leq T$ and all measurable functions $g$ such that $g(\zeta^T_u)$ is integrable.
	\end{theorem}
	\begin{proof}
		At $t=0$, the statement is self-evident. Now, let us consider the case where $t>0$. It follows from \eqref{eqmodification} that the set $\{\nu\leq t\}$ is $\sigma(\zeta^T_t)\vee \mathcal{N}_{\mathbb{P}}$-measurable. Hence, we have, $\mathbb{P}$-a.s.,
		\begin{equation}
			\mathbb{E}[g(\zeta^T_u)\mathbb{I}_{\{\nu\leq t\}}\vert\mathcal{F}^{\zeta^T}_t]=\mathbb{E}[g(\zeta^T_t)\mathbb{I}_{\{\nu\leq t\}}\vert\mathcal{F}^{\zeta^T}_t]=\mathbb{E}[g(\zeta^T_t)\mathbb{I}_{\{\nu\leq t\}}\vert\zeta^T_t]=\mathbb{E}[g(\zeta^T_u)\mathbb{I}_{\{\nu\leq t\}}\vert\zeta^T_t].
		\end{equation}
		It remains to show that, $\mathbb{P}$-a.s.,
		\begin{equation}
			\mathbb{E}[g(\zeta^T_u)\mathbb{I}_{\{t<\nu\}}\vert\mathcal{F}^{\zeta^T}_t]=\mathbb{E}[g(\zeta^T_u)\mathbb{I}_{\{t<\nu\}}\vert\zeta^T_t].
		\end{equation}
		Hence, it is sufficient to show that for all $\mathcal{F}^{\zeta^T}_t$ measurable $Y$ one has
		\begin{equation}
			\mathbb{E}[g(\zeta^T_u)\mathbb{I}_{\{t<\nu\}}Y]=\mathbb{E}[\mathbb{E}[g(\zeta^T_u)\vert\zeta^T_t]\mathbb{I}_{\{t<\nu\}}Y].\label{eqMarkovY}
		\end{equation}
		By the monotone class theorem it is sufficient to prove \eqref{eqMarkovY} for random variables $Y$ of the form 
		\begin{equation}
			Y=G(\zeta^T_t)L(\gamma_1,\ldots,\gamma_n)
		\end{equation}
		where 
		\begin{equation*}
			\gamma_k=\dfrac{\zeta^T_{t_k}}{t_k}-\dfrac{\zeta^T_{t_{k-1}}}{t_{k-1}},\,\,k=1,\ldots,n
		\end{equation*}
		$0<t_0<t_1\leq \cdots \leq t_n=t<u<T$ and $G$ and $L$ are bounded measurable functions.
		\begin{equation*}
			\alpha_k=\dfrac{W_{t_k}}{t_k}-\dfrac{W_{t_{k-1}}}{t_{k-1}},\,\,k=1,\ldots,n
		\end{equation*}
		we have $\gamma_k=\alpha_k$ on the set $\{t<\nu\}$, $k=1,\ldots,n$. Note that for $t<r$ and $z\in \mathbb{R}$, the Gaussian vectors $(\alpha_1,\ldots,\alpha_n)$ and $(W_t^{r,z},W_u^{r,z})$ are independent. Combining this, we have
		\begin{align*}
			\mathbb{E}[g(\zeta^T_u)\mathbb{I}_{\{t<\nu\}}G(\zeta^T_t)L(\gamma_1,\ldots,\gamma_n)]&=\mathbb{E}[g(\zeta^T_u)\mathbb{I}_{\{t<\nu\}}G(\zeta^T_t)L(\alpha_1,\ldots,\alpha_n)]\\
			&=\int_{\mathbb{R}}\int_{t}^{+\infty}\mathbb{E}[g(W_u^{r,z})G(W_t^{r,z})L(\alpha_1,\ldots,\alpha_n)]\mathbb{P}_{(\nu,Z)}(\mathrm{d}r,\mathrm{d}z)\\
			&=\int_{\mathbb{R}}\int_{t}^{+\infty}\mathbb{E}[g(W_u^{r,z})G(W_t^{r,z})]\mathbb{P}_{(\nu,Z)}(\mathrm{d}r,\mathrm{d}z)\mathbb{E}[L(\alpha_1,\ldots,\alpha_n)]\\
			&=\mathbb{E}[g(\zeta^T_u)G(\zeta^T_t)\mathbb{I}_{\{t<\nu\}}]\mathbb{E}[L(\alpha_1,\ldots,\alpha_n)]\\
			&=\mathbb{E}[g(\zeta^T_u)G(\zeta^T_t)\mathbb{I}_{\{t<\nu\}}]\mathbb{E}[L(\alpha_1,\ldots,\alpha_n)]\\
			&=\mathbb{E}[\mathbb{E}[g(\zeta^T_u)\vert\zeta^T_t]G(\zeta^T_t)\mathbb{I}_{\{t<\nu\}}]\mathbb{E}[L(\alpha_1,\ldots,\alpha_n)]\\
			&=\mathbb{E}[\mathbb{E}[g(\zeta^T_u)\vert\zeta^T_t]\mathbb{I}_{\{t<\nu\}}G(\zeta^T_t)L(\alpha_1,\ldots,\alpha_n)]\\
			&=\mathbb{E}[\mathbb{E}[g(\zeta^T_u)\vert\zeta^T_t]\mathbb{I}_{\{t<\nu\}}G(\zeta^T_t)L(\gamma_1,\ldots,\gamma_n)].
		\end{align*}
		This completes the proof.
	\end{proof}
	We  address the Bayesian estimates of $\nu$, $Z_T$, and $\zeta^T_u$ given the observation of the market information process $\zeta^T$ up to time $t<u$.
	\begin{proposition}\label{propconditional}
		For all $0<t<u<T$ and for every measurable function $g$ on $(0, T] \times \mathbb{R} \times \mathbb{R}$ such that $g\left(\nu, Z_T, \zeta^T_u\right)$ is integrable, we have, $\mathbb{P}$-a.s.,
		\begin{multline}
			\mathbb{E}[g(\nu,Z_T,\zeta^T_u)\vert\mathcal{F}_t^{\zeta^T}]=g(\nu,Z_T,\sigma Z_T)\mathbb{I}_{\{\nu\leq t\}}+\sum\limits _{i=1}^{l}\bigg[\dint_{t}^{u}g(r,z_i,\sigma z_i)\phi_{t}^{r,z_i}(\zeta^T_t)f(r,z_i)\mathrm{d}r\\
			+\dint_{u}^{T}\dint_{\mathbb{R}}g(r,z_i,y)p_{t,u}^{r,z_i}(\zeta^T_t,y)\mathrm{d}y\phi_{t}^{r,z_i}(\zeta^T_t)f(r,z_i)\mathrm{d} r\\+\dint_{\mathbb{R}}g(T,z_i,y)p_{t,u}^{T,z_i}(\zeta_t^T,y)\mathrm{d}y\,\phi_{\zeta_{t}^{T,z_i}}(\zeta_t^T)F(T,z_i)\bigg]p_i\mathbb{I}_{\{t<\nu\}},\label{eqlawoftauZgivenxi_t}
		\end{multline}
		where \begin{equation}
			F(T,z_i)=\displaystyle\int_{T}^{+\infty}f(r,z_i)\mathrm{d}r,
		\end{equation}
		\begin{align}
			\phi_{t}^{r,z}(x)
			&=\dfrac{\varphi_{t}^{r,z}(x)}{\sum\limits _{i=1}^{l}\Big[\int_{t}^{T}\varphi_{t}^{r,z_i}(\zeta^T_t)f(r,z_i)\mathrm{d}r+F(T,z_i)\varphi_{t}^{T,z_i}(\zeta^T_t)\Big]p_i}\mathbb{I}_{\{ t<r\}},
		\end{align}
		\begin{equation}
			\varphi_{t}^{r,z}(x)=p\bigg(\dfrac{t(r-t)}{r},x,\dfrac{t}{r}\sigma z\bigg) \,\,0<t<r,\label{eqvarphidensity}
		\end{equation}
		and 
		\begin{equation}
			p_{t,u}^{r,z}(x,y)=p\bigg(\dfrac{r-u}{r-t}(u-t),y,\dfrac{r-u}{r-t}x+\dfrac{u-t}{r-t}\sigma z\bigg)\mathrm{d}y,\,\,0<t<u<r.\label{eqtransitionp_{t,u}}
		\end{equation}
	\end{proposition}
	\begin{proof}
		On the set $\{\nu\leq t\}$ the statement is a consequence of the fact that $\nu$
		is an $\mathcal{F}_t^{\zeta^T}$-stopping time and that $\zeta^T_u=\zeta^T_t=\sigma Z_T$ on $\{\nu\leq t<u\}$. On the other hand, using the fact that $\zeta^T_{\nu}=\sigma Z_T$, we have $g\left(\nu, Z_T, \zeta^T_u\right)\mathbb{I}_{\{ t<\nu \}}$ is measurable with respect to $\sigma(\zeta^T_s, s\geq t)\vee \mathcal{N}_{\mathbb{P}}$. Thus, the Markov property yields
		\begin{equation}
			\mathbb{E}[g(\nu,Z_T,\zeta^T_u)\mathbb{I}_{\{ t<\nu \}}\vert\mathcal{F}_t^{\zeta^T}]=\mathbb{E}[g(\nu,Z_T,\zeta^T_u)\mathbb{I}_{\{ t<\nu \}}\vert\zeta^T_t].\label{eqEgnuZTzetat<nu}
		\end{equation}
		Therefore, we only need to compute the conditional law of $(\nu,Z_T,\zeta^T_u)$ with respect to $\zeta^T_t$. To do so, let us first compute the conditional law of $(\nu,Z_T)$ with respect to $\zeta^T_t$. It is not difficult to show that for every $B \in \mathcal{B}(\mathbb{R})$, we have
		\begin{equation}
			\mathbb{P}_{\zeta_t^T\vert\nu=r,Z_T=z}(B)=\dint_{B}q_t(r,z,x)\mu(\mathrm{d}x)
		\end{equation}
		where
		\begin{equation}
			q_t(x,r,z)=\sum\limits _{i=1}^l\left(\mathbb{I}_{\{x=\sigma z_i\}}\;\mathbb{I}_{\{z=z_i\}}\;\mathbb{I}_{\{r\leq t\}}\right)+\varphi_{t}^{r,z}(x)\;\prod\limits _{i= 1}^l\mathbb{I}_{\{x\neq \sigma z_i\}}\mathbb{I}_{\{ t<r\}}\label{eqq_t(r,x,z)}
		\end{equation}
		and
		\begin{equation}
			\mu(\mathrm{d}x)=\mathrm{d}x+\sum\limits _{i= 1}^l\delta_{\sigma z_i}(\mathrm{d}x).
		\end{equation}
		A direct application of Bayes formula yields that for every measurable function $g$ on $(0, T] \times \mathbb{R}$ such that $g\left(\nu, Z_T\right)$ is integrable, we have, $\mathbb{P}$-a.s.,
		\begin{small}
			\begin{multline}
				\mathbb{E}[g(\nu,Z_T)\vert\zeta^T_{t}]=\dfrac{\dint_{\mathbb{R}}\dint_{0}^{T}g(r,z)q_{t}(\zeta_t^T,r,z)\mathbb{P}_{(\nu,Z_T)}(\mathrm{d}r,\mathrm{d}z)}{\dint_{\mathbb{R}}\dint_{0}^{T}q_{t}(\zeta_t^T,r,z)\mathbb{P}_{(\nu,Z_T)}(\mathrm{d}r,\mathrm{d}z)}=\sum\limits _{i= 1}^l\dfrac{\dint_{0}^{t}g(r,z_i)\;\mathbb{P}_{\nu\vert Z_T=z_i}(\mathrm{d} r)}{\mathbb{P}(\nu\leq t\vert Z_T=z_i)}\;\mathbb{I}_{\{\zeta_t^T=\sigma z_i\}}\\
				+\sum\limits _{i= 1}^l\dfrac{\dint_{t}^{T}g(r,z_i)\varphi_{t}^{r,z_i}(\zeta_t^T)\mathbb{P}_{\nu\vert Z_T=z_i}(\mathrm{d} r)}{\sum\limits _{i=1}^l\dint_{t}^{T}\varphi_{t}^{r,z_i}(\zeta_t^T)\mathbb{P}_{\nu\vert Z_T=z_i}(\mathrm{d}r)p_i}p_i\prod\limits _{i=1}^l\mathbb{I}_{\{\zeta_t^T\neq \sigma z_i\}}.\label{eqbayestau,Zgivenzeta}
			\end{multline}
		\end{small}
		We split the expectation on the right hand side of \eqref{eqEgnuZTzetat<nu} into two expectations as follows:
		\begin{equation}
			\mathbb{E}[g(\nu,Z_T,\zeta^T_u)\mathbb{I}_{\{ t<\nu \}}\vert\zeta^T_t]=\mathbb{E}[g(\nu,Z_T,\sigma Z_T)\mathbb{I}_{\{ t<\nu\leq u \}}\vert\zeta^T_t]+\mathbb{E}[g(\nu,Z_T,\zeta^T_u)\mathbb{I}_{\{ u<\nu \}}\vert\zeta^T_t].
		\end{equation}
		The first expectation can be computed using \eqref{eqbayestau,Zgivenzeta},
		\begin{multline}
			\mathbb{E}[g(\nu,Z_T,\sigma Z_T)\mathbb{I}_{\{ t<\nu\leq u \}}\vert\zeta^T_t]=\sum\limits _{i= 1}^l\dfrac{\dint_{t}^{u}g(r,z_i,\sigma z_i)\varphi_{t}^{r,z_i}(\zeta_t^T)\mathbb{P}_{\nu\vert Z_T=z_i}(\mathrm{d} r)}{\sum\limits _{i=1}^l\dint_{t}^{T}\varphi_{t}^{r,z_i}(\zeta_t^T)\mathbb{P}_{\nu\vert Z_T=z_i}(\mathrm{d}r)p_i}p_i\prod\limits _{k=1}^l\mathbb{I}_{\{\zeta_t^T\neq \sigma z_k\}}.\label{eqeqwithoutcond1}
		\end{multline}
		Concerning the second expectation, we proceed as follows 
		\begin{align}
			\mathbb{E}[g(\nu,Z_T,\zeta_u^T)&\mathbb{I}_{\{ u<\nu \}}\vert\zeta_t^T=x]\nonumber\\
			&=\dint_{\mathbb{R}}\dint_{u}^{T}\mathbb{E}[g(\nu,Z_T,\zeta_u^T)\vert\zeta_t^T=x,\nu=r,Z_T=z]\,\mathbb{P}(\nu \in \mathrm{d}r,Z_T\in \mathrm{d}z|\zeta_t^T=x)\nonumber\\
			&=\dint_{\mathbb{R}}\dint_{u}^{T}\mathbb{E}[g(r,z,W_u^{r,z})|W_t^{r,z}=x]\mathbb{P}(\nu \in \mathrm{d}r,Z_T\in \mathrm{d}z\vert\zeta_t^T=x)\nonumber\\
			&=\sum\limits _{i= 1}^l\dfrac{\dint_{u}^{T}\dint_{\mathbb{R}}g(r,z_i,y)p_{t,u}^{r,z_i}(x,y)\mathrm{d}y\,\varphi_{t}^{r,z_i}(x)\mathbb{P}_{\nu\vert Z_T=z_i}(\mathrm{d} r)}{\sum\limits _{i=1}^l\dint_{t}^{T}\varphi_{t}^{r,z_i}(x)\mathbb{P}_{\nu\vert Z_T=z_i}(\mathrm{d}r)p_i}p_i\prod\limits _{k=1}^l\mathbb{I}_{\{x\neq \sigma z_k\}}\label{eqeqwithoutcond2}
		\end{align}
		On the other hand, since $\nu=\tau\wedge T$, the conditional law $\mathbb{P}_{\nu\vert Z_T=z_i}$ of $\nu$ given $Z_T=z_i$ is expressed as
		\begin{equation}
			\mathbb{P}_{\nu\vert Z_T=z_i}(\mathrm{d}r)=\Big[f(r,z_i)\mathbb{I}_{\{r<T\}}+\displaystyle\int_{T}^{+\infty}f(r,z_i)\mathrm{d}r\,\mathbb{I}_{\{r=T\}}\Big](\mathrm{d}r+\delta_{T}(\mathrm{d}r)).\label{eqlawofnu}
		\end{equation}
		Hence, by combining \eqref{eqlawofnu} with the fact that
		\begin{equation}
			\prod\limits _{k=1}^l\mathbb{I}_{\{\zeta_t^T\neq \sigma z_k\}}=\mathbb{I}_{\{t<\nu\}},
		\end{equation}
		we obtain, $\mathbb{P}$-a.s.,
		\begin{multline}
			\mathbb{E}[g(\nu,Z_T,\sigma Z_T)\mathbb{I}_{\{ t<\nu\leq u \}}\vert\zeta^T_t]=\sum\limits _{i= 1}^l\dint_{t}^{u}g(r,z_i,\sigma z_i)\phi_{\zeta_{t}^{r,z_i}}(\zeta_t^T)f(r,z_i)\mathrm{d} r\,p_i\prod\limits _{k=1}^l\mathbb{I}_{\{\zeta_t^T\neq \sigma z_k\}}\label{eqeq1}
		\end{multline}
		and
		\begin{multline}
			\mathbb{E}[g(\nu,Z_T,\zeta_u^T)\mathbb{I}_{\{ u<\nu \}}\vert\zeta_t^T]=\sum\limits _{i= 1}^l\bigg[\dint_{u}^{T}\dint_{\mathbb{R}}g(r,z_i,y)p_{t,u}^{r,z_i}(\zeta_t^T,y)\mathrm{d}y\,\phi_{\zeta_{t}^{r,z_i}}(\zeta_t^T)f(r,z_i)\mathrm{d}r\\
			+\dint_{\mathbb{R}}g(T,z_i,y)p_{t,u}^{T,z_i}(\zeta_t^T,y)\mathrm{d}y\,\phi_{\zeta_{t}^{T,z_i}}(\zeta_t^T)F(T,z_i)\bigg]p_i\prod\limits _{k=1}^l\mathbb{I}_{\{\zeta_t^T\neq \sigma z_k\}}.\label{eqeq2}
		\end{multline}
		This completes the proof.
	\end{proof}
	\begin{remark}
		Following the same analysis as the proof of Theorem 4.3 in \cite{L}, we may conclude that the market filtration $\mathbb{F}^{\zeta^T}$ satisfies the usual conditions of right-continuity and completeness.
	\end{remark}
	Another important property of the market information process $\zeta^T$ is its semi-martingale decomposition:
	\begin{theorem}\label{thmsemimartingale}
		The market information process $\zeta^T$ is a semi-martingale with respect to $\mathbb{F}^{\zeta^T}$. Moreover, its decomposition is given by
		\begin{equation}
			\zeta^T_{t}=I_t+\sum\limits_{i=1}^{l}\int_{0}^{t} \mathfrak{f}_i(t,\zeta^T_{t})\mathbb{I}_{\{t<\nu\}}\mathrm{d}t,\label{eqdynamicszeta}
		\end{equation}
		where, $I$ is a Brownian motion stopped at $\nu$ and the functions $(\mathfrak{f}_i)_{1\leq i\leq l}$ are given by
		\begin{small}
			\begin{equation}
				\mathfrak{f}_i(t,x)=\dfrac{\dint_{t}^{T}\dfrac{\sigma z_i-x}{r-t}\dfrac{p(r-t,\sigma z_i-x)}{p(r,\sigma z_i)}f(r,z_i)\mathrm{d}r+F(T,z_i)\dfrac{\sigma z_i-x}{T-t}\dfrac{p(T-t,\sigma z_i-x)}{p(T,\sigma z_i)}}{\displaystyle\sum\limits _{i=1}^{l}\bigg[\dint_{t}^{T}\dfrac{p(r-t,\sigma z_i-x)}{p(r,\sigma z_i)}f(r,z_i)\mathrm{d}r+F(T,z_i)\dfrac{p(T-t,\sigma z_i-x)}{p(T,\sigma z_i)}\bigg]p_i}\,p_i.\label{eqmathfrakf_i}
			\end{equation}
		\end{small}
	\end{theorem}
	\begin{proof}
		The proof of the current theorem closely parallels the method used in the proof of Theorem 4.9 in \cite{L}.
	\end{proof}
	\section{Price process}
	\quad\,\,In this section, we derive an expression for the value process of a contract that delivers the cash flow $Z_T$ at time $T$, we provide estimates of the a priori unknown bankruptcy time, based on the observations of the market information process $\zeta^T$ up to time $t$. Subsequently, we analyze the dynamics of the price process and demonstrate that  the price process satisfies a diffusion equation.
	\begin{proposition}
		The price of the cash flow $Z_T$ is given by
		\begin{small}
			\begin{align}
				&B_t^T=P_t^T\,\mathbb{E}[Z_T\vert \mathcal{F}_t^{\zeta^T}]=\exp\Big(-\int_t^Tr(u)\mathrm{d}u\Big)\,Z_T\,\mathbb{I}_{\{\nu\leq t\}}\nonumber\\&+\exp\Big(-\int_t^Tr(u)\mathrm{d}u\Big)\displaystyle\sum\limits _{i=1}^{l}z_i\dfrac{\displaystyle\int_{t}^{T}\varphi_{t}^{r,z_i}(\zeta^T_t)f(r,z_i)\mathrm{d}r+F(T,z_i)\varphi_{t}^{T,z_i}(\zeta^T_t)}{\displaystyle\sum\limits _{i=1}^{l}\bigg[\displaystyle\int_{t}^{T}\varphi_{t}^{r,z_i}(\zeta^T_t)f(r,z_i)\mathrm{d}r+F(T,z_i)\varphi_{t}^{T,z_i}(\zeta^T_t)\bigg]p_i}p_i\,\mathbb{I}_{\{t<\nu\}}.\label{eqpricingformulla}
			\end{align}
		\end{small}
	\end{proposition}
	\begin{proof}
		The result is a consequence of Proposition \ref{propconditional}.
	\end{proof}
	\begin{proposition}
		Let $0<t<T$, $g: [0, T] \rightarrow \mathbb{R}$ be a Borel function such that $\mathbf{E}[\vert g(\nu) \vert]<+\infty$. Then, $\mathbb{P}$-a.s.,
		\begin{multline}
			\mathbb{E}[g(\nu)\vert \mathcal{F}_t^{\zeta^T}]=g(\nu)\,\mathbb{I}_{\{\nu\leq t\}}\\+\displaystyle\sum\limits _{i=1}^{l}\dfrac{\displaystyle\int_{t}^{T}g(r)\varphi_{t}^{r,z_i}(\zeta^T_t)f(r,z_i)\mathrm{d}r+F(T,z_i)g(T)\varphi_{t}^{T,z_i}(\zeta^T_t)}{\displaystyle\sum\limits _{i=1}^{l}\bigg[\displaystyle\int_{t}^{T}\varphi_{t}^{r,z_i}(\zeta^T_t)f(r,z_i)\mathrm{d}r+F(T,z_i)\varphi_{t}^{T,z_i}(\zeta^T_t)\bigg]\,p_i}\,p_i\,\mathbb{I}_{\{t<\nu\}}.\label{eqnugivenFt}
		\end{multline}
	\end{proposition}
	\begin{proof}
		The result follows from the implications of Proposition \ref{propconditional}.
	\end{proof}
	In the following result, we consider the problem of pricing a European option on the price $B_{t}^T$ at time $t<T$.
	\begin{proposition}
		For a strike price $K$, the price at time $0$ of a $t$-maturity call option on $B_{t}^T$ is given by
		\begin{multline}
			C_{0}^{t}=P_{0}^{t}\, \mathbb{E}[\left(B_{t}^{T}-K\right)^{+}]=P_0^t\displaystyle\sum\limits _{i=1}^{l} (P_t^Tz_i-K)^+\int_{0}^{t}f(r,z_i)\mathrm{d}r\,p_i\\
			+P_0^t\,\dint_{\mathbb{R}}\bigg(\displaystyle\sum\limits _{i=1}^{l}(P_t^T\,z_i-K)\bigg[\displaystyle\int_{t}^{T}\varphi_{t}^{r,z_i}(x)f(r,z_i)\mathrm{d}r+F(T,z_i)\varphi_{t}^{T,z_i}(x)\bigg]\,p_i\bigg)^+\mathrm{d}x.
		\end{multline}
	\end{proposition}
	\begin{proof}
		We have
		\begin{equation}
			C_{0}^{t}=P_{0}^{t}\, \mathbb{E}[\left(B_{t}^{T}-K\right)^{+}].\label{eqC0}
		\end{equation}
		We split the above expectation as follows
		\begin{equation}
			\mathbb{E}[(B_t^T-K)^+]=\mathbb{E}[(B_t^T-K)^+\mathbb{I}_{\{ \nu\leq t\}}]+\mathbb{E}[(B_t^T-K)^+\mathbb{I}_{\{t< \nu\}}].\label{eqC0tauHTsplit}
		\end{equation}
		For the first expectation, using \eqref{eqpricingformulla} we obtain
		\begin{align}
			\mathbb{E}[(B_t^T-K)^+\mathbb{I}_{\{ \nu\leq t\}}]&=\mathbb{E}[(P_t^T\,Z_T-K)^+\mathbb{I}_{\{ \tau\leq t\}}]\nonumber\\
			&=\displaystyle\sum\limits _{i=1}^{l} (P_t^Tz_i-K)^+\mathbb{P}(\tau\leq t\vert Z_T=z_i)\,p_i\nonumber\\
			&=\displaystyle\sum\limits _{i=1}^{l} (P_t^Tz_i-K)^+\int_{0}^{t}f(r,z_i)\mathrm{d}r\,p_i.\label{eqexpectation1}
		\end{align}
		For the second expectation, using again \eqref{eqpricingformulla} and the formula of total probability we obtain
		\begin{align}
			\mathbb{E}[(B_t^T&-K)^+\mathbb{I}_{\{t< \nu\}}]\nonumber\\&=\mathbb{E}\bigg[\bigg(\displaystyle\sum\limits _{i=1}^{l}(P_t^T\,z_i-K)\dfrac{\displaystyle\int_{t}^{T}\varphi_{t}^{r,z_i}(\zeta^T_t)f(r,z_i)\mathrm{d}r+F(T,z_i)\varphi_{t}^{T,z_i}(\zeta^T_t)}{\displaystyle\sum\limits _{i=1}^{l}\bigg[\displaystyle\int_{t}^{T}\varphi_{t}^{r,z_i}(\zeta^T_t)f(r,z_i)\mathrm{d}r+F(T,z_i)\varphi_{t}^{T,z_i}(\zeta^T_t)\bigg]p_i}\,p_i\bigg)^+\mathbb{I}_{\{t< \nu\}}\bigg]\nonumber\\
			&=\dint_{\mathbb{R}}\bigg(\displaystyle\sum\limits _{i=1}^{l}(P_t^T\,z_i-K)\bigg[\displaystyle\int_{t}^{T}\varphi_{t}^{r,z_i}(x)f(r,z_i)\mathrm{d}r+F(T,z_i)\varphi_{t}^{T,z_i}(x)\bigg]\,p_i\bigg)^+\mathrm{d}x.\label{eqexpectation2}
		\end{align}
		This completes the proof.
	\end{proof}
	In the following result, we provide the dynamics of the price process $B^T=(B_{t}^T, t\leq T)$.
	\begin{proposition}\label{propdynamicofZ}
		For $0<t < \nu$, the dynamics of the price process $B^T$ are given by
		\begin{equation}
			\mathrm{d}B_{t}^T=\sum\limits_{i=1}^l(P_t^Tz_i-B_{t}^T)\mathfrak{f}_i(t,\zeta^T_{t})\mathrm{d}I_t+r_tB_{t}^T\mathrm{d}t.\label{eqpricedynamic}
		\end{equation}
		For $t\geq \nu$, $B_{t}^T=P_t^TZ_T$, and for $t=0$, $B_{0}^T=P_0^T\sum\limits _{i=1}^{l}\,z_i\,p_i$. Here, $I$ is an $\mathbb{F}^{\zeta^T}$-Brownian motion stopped at $\nu$ and the functions $(\mathfrak{f}_i)_{1\leq i\leq l}$ are defined in \eqref{eqmathfrakf_i}.
	\end{proposition}
	\begin{proof}
		The proof is based on It\^o's Formula. We consider the functions 
		\begin{equation*}
			t\in (0, T)\longrightarrow \bar{f}_i(t,x):=\dint_{0}^{T}\dfrac{p(r-t,\sigma z_i-x)}{p(r,\sigma z_i)}f(r,z_i)\mathbb{I}_{\{t<r\}}\mathrm{d}r\,p_i,\,\,1\leq i\leq l.
		\end{equation*}
		We have for all $t> 0$,
		\begin{align*}
			\dint_{\mathbb{R}}\bar{f}_i(t,x)p(t,x)\mathrm{d}x&=\dint_{t}^{T}\dint_{\mathbb{R}}\dfrac{p(r-t,\sigma z_i-x)p(t,x)}{p(r,\sigma z_i)}\mathrm{d}xf(r,z_i)\mathrm{d}r\,p_i\\
			&=\dint_{t}^{T}\dint_{\mathbb{R}}\varphi_{t}^{r,z_i}(x)\mathrm{d}xf(r,z_i)\mathrm{d}r\,p_i\\
			&=\dint_{t}^{T}f(r,z_i)\mathrm{d}r\,p_i\leq 1.
		\end{align*}
		Hence,
		\begin{equation}
			\dint_{0}^{T}\dfrac{p(r-t,\sigma z_i-x)}{p(r,\sigma z_i)}f(r,z_i)\mathbb{I}_{\{t<r\}}\mathrm{d}r\,p_i<+\infty,\,\,\text{for}\,\text{a.e.}\,x.\label{eqintegrable}
		\end{equation}
		On the other hand, if $x\neq \sigma z_i$, it is easy to see that the function
		\begin{equation*}
			t\in (0, T)\longrightarrow f_i(t,x,r):=\dfrac{p(r-t,\sigma z_i-x)}{p(r,\sigma z_i)}\mathbb{I}_{\{t<r\}}\,p_i
		\end{equation*}
		is continuous on $(0, T)$ and
		\begin{equation*}
			\lim\limits_{h\uparrow 0}\dfrac{1}{h}\dfrac{p(-h,\sigma z_i-x)}{p(r,\sigma z_i)}=0,
		\end{equation*}
		which implies that if $x\neq \sigma z_i$, the function $t\longrightarrow f_i(t,x,r)$ is differentiable. Under the same condition, $x\neq \sigma z_i$, we have 
		\begin{equation}
			t\longrightarrow\partial_tf_i(t,x,r)=\dfrac{1}{2}\bigg(\dfrac{1}{r-t}-\dfrac{(\sigma z_i-x)^2}{(r-t)^2}\bigg)f_i(t,x,r)\label{eqpartialtexpression}
		\end{equation}
		is continuous on $(0, +\infty)$ and the function $r\longrightarrow\partial_tf_i(t,x,r)$	is continuous on $[0, T]$. This implies that for $x\neq \sigma z_i$, we have
		\begin{equation}
			\partial_t\bar{f}_i(t,x)=\dint_{0}^{T}\partial_tf_i(t,x,r)f(r,z_i)\mathrm{d}r\,p_i.\label{eqpartialt}
		\end{equation}
		Similarly, using the fact that $x\longrightarrow f_i(t,x,r)$ is differentiable on $\mathbb{R}$, and that the functions 
		\begin{equation}
			x\longrightarrow\partial_xf_i(t,x,r)=\dfrac{\sigma z_i-x}{r-t}f_i(t,x,r)\label{eqpartialxexpression}
		\end{equation}
		and $r\longrightarrow\partial_xf_i(t,x,r)$ are continuous on $\mathbb{R}$ and $[0, T]$, respectively, we obtain that 
		\begin{equation}
			\partial_x\bar{f}_i(t,x)=\dint_{0}^{T}\partial_xf_i(t,x,r)f(r,z_i)\mathrm{d}r\,p_i.\label{eqpartailx}
		\end{equation}
		Similarly, we obtain 
		\begin{equation}
			\partial^2_x\bar{f}_i(t,x)=\dint_{0}^{T}\partial^2_xf_i(t,x,r)f(r,z_i)\mathrm{d}r\,p_i\label{eqpartialxx}
		\end{equation}
		due to the fact that $x\longrightarrow\partial_xf_i(t,x,r)$ is differentiable on $\mathbb{R}$, and that the functions 
		\begin{equation}
			x\longrightarrow\partial^2_xf_i(t,x,r)=-\bigg(\dfrac{1}{r-t}-\dfrac{(\sigma z_i-x)^2}{(r-t)^2}\bigg)f_i(t,x,r),\label{eqpartialxxexpression}
		\end{equation}
		and $r\longrightarrow\partial^2_xf_i(t,x,r)$ are continuous on $\mathbb{R}$ and $[0, T]$, respectively. From \eqref{eqpricingformulla}, we have
		\begin{align}
			B_{t}^T&=P_t^T\sum\limits _{i=1}^{l}\,z_i\mathfrak{g}_i(t,\zeta^T_{t})\,\mathbb{I}_{\{t<\nu\}}+P_t^T\,Z_T\,\mathbb{I}_{\{\nu\leq t\}},
		\end{align}
		where,
		\begin{equation}
			\mathfrak{g}_i(t,x)=\dfrac{\bar{f}_i(t,x)+F(T,z_i)f_i(t,x,T)}{\displaystyle\sum\limits _{i=1}^{l}\Big[\bar{f}_i(t,x)+F(T,z_i)f_i(t,x,T)\Big]}.
		\end{equation}
		It follows from Theorem \ref{thmsemimartingale} that $\zeta^T$ has the following dynamics
		\begin{equation}
			\mathrm{d}\zeta^T_{t}=\mathrm{d}I_t+\sum\limits_{i=1}^{l} \mathfrak{f}_i(t,\zeta^T_{t})\mathbb{I}_{\{t<\nu\}}\mathrm{d}t,\label{eqdynamicszeta}
		\end{equation}
		observing that $\zeta^T_{t\wedge \nu}=\zeta^T_{t}$ and that for $0<t<\nu$ we have
		\begin{align}
			B_{t}^T&=P_t^T\sum\limits _{i=1}^{l}\,z_i\mathfrak{g}_i(t,\zeta^T_{t}).\label{eqB_tT}
		\end{align}
		It follows from \eqref{eqpartialt} that for $x\notin\{\sigma z_1, \ldots, \sigma z_n \}$
		\begin{multline}
			\partial_t\mathfrak{g}_i(t,x)=\dfrac{\partial_t\bar{f}_i(t,x)+F(T,z_i)\partial_tf_i(t,x,T)}{\displaystyle\sum\limits _{i=1}^{l}\Big[\bar{f}_i(t,x)+F(T,z_i)f_i(t,x,T)\Big]}\\-\mathfrak{g}_i(t,x)\displaystyle\sum\limits _{j=1}^{l}\dfrac{\partial_t\bar{f}_j(t,x)+F(T,z_i)\partial_tf_j(t,x,T)}{\displaystyle\sum\limits_{i=1}^{l}\Big[\bar{f}_i(t,x)+F(T,z_i)f_i(t,x,T)\Big]},
		\end{multline}
		using \eqref{eqpartailx} we have
		\begin{multline}
			\partial_x\mathfrak{g}_i(t,x)=\dfrac{\partial_x\bar{f}_i(t,x)+F(T,z_i)\partial_xf_i(t,x,T)}{\displaystyle\sum\limits _{i=1}^{l}\Big[\bar{f}_i(t,x)+F(T,z_i)f_i(t,x,T)\Big]}\\-\mathfrak{g}_i(t,x)\displaystyle\sum\limits _{j=1}^{l}\dfrac{\partial_x\bar{f}_j(t,x)+F(T,z_i)\partial_xf_j(t,x,T)}{\displaystyle\sum\limits _{i=1}^{l}\Big[\bar{f}_i(t,x)+F(T,z_i)f_i(t,x,T)\Big]}.
		\end{multline}
		Hence,
		\begin{equation}
			\partial_x\mathfrak{g}_i(t,x)=\mathfrak{f}_i(t,x)-\mathfrak{g}_i(t,x)\sum\limits _{j=1}^{l}\mathfrak{f}_j(t,x).\label{eqpartialxf^1}
		\end{equation}
		From \eqref{eqpartialxx} and \eqref{eqpartailx} we have
		\begin{align}
			\partial_x\mathfrak{f}_i(t,x)&=\dfrac{\partial_x^2\bar{f}_i(t,x)+F(T,z_i)\partial_x^2f_i(t,x,T)}{\displaystyle\sum\limits _{i=1}^{l}\Big[\bar{f}_i(t,x)+F(T,z_i)f_i(t,x,T)\Big]}-\mathfrak{f}_i(t,x)\sum\limits _{j=1}^{l}\mathfrak{f}_j(t,x),
		\end{align}
		thus, \eqref{eqpartialxf^1} implies that
		\begin{align}
			&\partial^2_x\mathfrak{g}_i(t,x)=\dfrac{\partial_x^2\bar{f}_i(t,x)+F(T,z_i)\partial_x^2f_i(t,x,T)}{\displaystyle\sum\limits _{i=1}^{l}\Big[\bar{f}_i(t,x)+F(T,z_i)f_i(t,x,T)\Big]}-2\mathfrak{f}_i(t,x)\sum\limits _{j=1}^{l}\mathfrak{f}_j(t,x)\nonumber\\
			&-\mathfrak{g}_i(t,x)\displaystyle\sum\limits _{i=1}^{l}\dfrac{\partial_x^2\bar{f}_i(t,x)+F(T,z_i)\partial_x^2f_i(t,x,T)}{\displaystyle\sum\limits _{i=1}^{l}\Big[\bar{f}_i(t,x)+F(T,z_i)f_i(t,x,T)\Big]}+2\mathfrak{g}_i(t,x)\bigg[\sum\limits _{j=1}^{l}\mathfrak{f}_j(t,x)\bigg]^2.\label{eqpartialxxf^1}
		\end{align}
		From \eqref{eqpartialtexpression} and \eqref{eqpartialxxexpression} we observe that 
		\begin{equation}
			\partial^2_xf_i(t,x,T)=-2\partial_tf_i(t,x,T),\label{erpartailxx=partialt}
		\end{equation}
		and 
		\begin{equation}
			\partial^2_x\bar{f}_i(t,x)=-2\partial_t\bar{f}_i(t,x),\label{erpartailxx=partialt}
		\end{equation}
		then \eqref{eqpartialxxf^1} becomes
		\begin{equation}
			\partial^2_x\mathfrak{g}_i(t,x)=-2\partial_t\mathfrak{g}_i(t,x)-2\mathfrak{f}_i(t,x)\sum\limits _{j=1}^{l}\mathfrak{f}_j(t,x)+2\mathfrak{g}_i(t,x)\bigg[\sum\limits _{j=1}^{l}\mathfrak{f}_j(t,x)\bigg]^2.\label{eqpartialxxf^1t}
		\end{equation}
		Note that, $\mathbb{P}$-a.s.,
		\begin{equation}
			\{t<\nu\}=\bigcap\limits_{1\leq i \leq l}\{\zeta^T_t\neq \sigma z_i\},\label{eqzeta^TneqsigmaZ}
		\end{equation}
		which implies $\zeta^T_{t}\neq \sigma z_i$ for all $i\in\{1,\ldots,l\}$ on the set $\{t<\nu\}$. Therefore, by applying the It\^o formula, we deduce that for $t<\nu$, we have
		\begin{equation}
			\mathrm{d}\mathfrak{g}_i(t,\zeta^T_{t})=\bigg[\partial_t\mathfrak{g}_i(t,\zeta^T_{t})+\partial_x\mathfrak{g}_i(t,\zeta^T_{t})\sum\limits _{j=1}^{l}\mathfrak{f}_j(t,\zeta^T_{t})+\dfrac{1}{2}\partial^2_x\mathfrak{g}_i(t,\zeta^T_{t})\bigg]\mathrm{d}t+\partial_x\mathfrak{g}_i(t,\zeta^T_{t})\mathrm{d}I_t.\label{eqIto}
		\end{equation}
		Inserting \eqref{eqpartialxf^1} and \eqref{eqpartialxxf^1t} into \eqref{eqIto} we obtain that 
		\begin{equation}
			\mathrm{d}\mathfrak{g}_i(t,\zeta^T_{t})=\bigg[\mathfrak{f}_i(t,\zeta^T_{t})-\mathfrak{g}_i(t,\zeta^T_{t})\sum\limits _{j=1}^{l}\mathfrak{f}_j(t,\zeta^T_{t})\bigg]\mathrm{d}I_t.\label{eqdynamicf^1}
		\end{equation}
		We conclude from \eqref{eqB_tT} that for $t<\nu$, we have 
		\begin{align}
			\mathrm{d}B_{t}^T&=P_t^T\sum\limits _{i=1}^{l}\,z_i\mathrm{d}\mathfrak{g}_i(t,\zeta^T_{t})+r_t\,P_t^T\sum\limits _{i=1}^{l}\,z_i\mathfrak{g}_i(t,\zeta^T_{t})\mathrm{d}t\nonumber\\
			&=\sum\limits_{i=1}^l(P_t^T\,z_i-B_{t}^T)\mathfrak{f}_i(t,\zeta^T_{t})\mathrm{d}I_t+r_t\,B_t^T\mathrm{d}t.
		\end{align}
		It follows from \eqref{eqpricingformulla} that for $\nu\leq t$ we have $B_{t}^T=P_t^T\,Z_T$ and for $t=0$, we have $B_{t}^T=P_t^T\sum\limits _{i=1}^{l}\,z_i\,p_i$.
	\end{proof}
	\section{Compensator of certain special processes with respect to $\mathbb{F}^{\zeta^T}$}
	\label{sect3}
	\quad\,\,Certainly, market agents prioritize obtaining as much information as possible about the bankruptcy time $\tau$. However, to compile certain facts about the bankruptcy time $\tau$, it is crucial to determine the nature of the stopping time $\nu=\tau\wedge T$. Market agents can anticipate the occurrence of bankruptcy if $\nu$ is predictable, whereas bankruptcy takes place unexpectedly if $\nu$ is totally inaccessible. In this section, to gain deeper insights into $\nu$, we explicitly compute the compensator of $\nu$ and leverage a well-known equivalence between the categories of stopping times and the regularity of their compensators.
	
	Consider a finite variation process $V$ with $V_0 = 0$ and locally integrable total variation. The compensator of $V$ is a unique finite variation predictable process, denoted as $A$, such that the process $V-A$ is a local martingale. If $V$ is an increasing process, it is a sub-martingale. Consequently, according to the Doob-Meyer theorem, the compensator $A$ is also increasing. Moreover, when $V$ is a càdlàg, adapted, and locally integrable increasing process, its compensator $A$ remains increasing and, additionally, for any stopping time $S$ and non-negative predictable processes $U$, the following equality holds:
	\begin{equation}
		\mathbb{E}\bigg[\displaystyle\int_0^{S}U_s\mathrm{d}V_s \bigg]=\mathbb{E}\bigg[\displaystyle\int_0^{S}U_s\mathrm{d}A_s \bigg].
	\end{equation}
	Furthermore, the compensator $A$ of $V$ is the unique right-continuous predictable and increasing process with $A_0=0$ which satisfies
	\begin{equation}
		\mathbb{E}\bigg[\displaystyle\int_0^{\infty}U_s\mathrm{d}V_s \bigg]=\mathbb{E}\bigg[\displaystyle\int_0^{\infty}U_s\mathrm{d}A_s \bigg]
	\end{equation}
	for all non-negative predictable $U$. See for instance \cite{M} and \cite{N}. This section aims to give the explicit computation of the compensator of the random time $\nu$, that is, the compensator of the $\mathbb{F}^{\zeta^T}$-sub-martingale:
	\begin{equation}
		H_t=\mathbb{I}_{\{ \nu\leq t \}},\,\, t\geq 0,
	\end{equation}
	which is defined as the unique adapted, natural, increasing, integrable process $K$ satisfying
	\begin{equation}
		H=N+K,\label{eqcompensator}
	\end{equation}
	Here $N$ represents a right-continuous martingale. In our analysis, we employ the methodology developed by P.A. Meyer \cite{M} for computing the compensator of a sub-martingale. This method is then adapted to determine the compensator of the process $H$. For the reader's convenience, we revisit pertinent materials and definitions outlined in \cite{M}. 
	Consider a filtration $\mathbb{F}=(\mathcal{F}_t)_{t\geq 0}$ satisfying the usual hypothesis of right-continuity and completeness. Let $X$ be a right-continuous $\mathbb{F}$-super-martingale, and denote $\mathcal{T}$ as the collection of all finite $\mathbb{F}$-stopping times relative to this family. We classify the process $X$ as belonging to the class (D) if the collection of random variables $(X_S, S\in \mathcal{T})$ is uniformly integrable. Moreover, we define the right-continuous super-martingale $X$ as a potential if the random variables $X_t$ are non-negative and if
	\begin{equation}
		\lim\limits_{t \rightarrow +\infty}\mathbb{E}[X_t]=0.
	\end{equation}
	If $C = (C_t, t \geq 0)$ is an integrable increasing process, and $L = (L_t, t \geq 0)$ represents the right-continuous modification of the martingale $(\mathbb{E}[C_{\infty}\vert \mathcal{F}_t], t \geq 0)$, then the process $Y=(Y_t, t \geq 0)$, defined by
	\begin{equation}
		Y_t= L_t - C_t
	\end{equation}
	This quantity is denoted as the potential generated by $C$. The methodology employed in this context relies on the convergence in the weak topology $\sigma(L^1,L^{\infty})$. It's valuable to recall the definition of this convergence. Let $(Y_n){n\in \mathbb{N}}$ be a sequence of integrable real-valued random variables. The sequence $(Y_n)_{n\in \mathbb{N}}$ is said to converge to an integrable random variable $Y$ in the weak topology $\sigma(L^1,L^{\infty})$ if 
	\begin{equation}
		\lim\limits_{n\rightarrow +\infty}\mathbb{E}[Y_n\eta]=\mathbb{E}[Y\eta],\,\,\text{for all}\, \eta\in L^{\infty}(\mathbb{P}).
	\end{equation}
	In the subsequent discussion, we delve into the representation of the compensator of $H$ which notably involves the local time of $\zeta^T$. Therefore, before presenting the main result of this section, we establish specific properties of the local time. It is a well-known that for a continuous semi-martingale, the local time can be defined. Tanaka's formula, in particular, provides a definition of local time for arbitrary continuous semi-martingales. Since our process $\zeta^T$ is a continuous semi-martingale, the local time $(L^{\zeta^T}(t,x), t\geq 0)$ of $\zeta^T$ at level $x\in \mathbb{R}$ is well-defined.\\
	In the following proposition, we address the continuity and boundedness of the local time $L^{\zeta^T}(t,x)$.
	\begin{proposition}\label{proplocaltime}
		Let $x\in \mathbb{R}$, and $(L^{\zeta^T}(t,x), t\geq 0)$ be the local time of $\zeta^T$ at level $x$. We have:
		\begin{enumerate}
			\item[(i)] There is a version of $L^{\zeta^T}(t,x)$ such that the map $(t,x)\in\mathbb{R}_{+}\times \mathbb{R}\rightarrow L^{\zeta^T}(t,x)$ is continuous, $\mathbb{P}$-a.s.
			\item[(ii)] For every continuous function $g$ on $\mathbb{R}$, the function $x\in\mathbb{R}\rightarrow g(x)\,L^{\zeta^T}(t,x)$ is bounded for all $t\geq 0$, $\mathbb{P}$-a.s.~ (the bound may depend on $t$ and $\omega$).
			\item[(iii)] Let $(x_n)_{n\in \mathbb{N}}$ be a sequence in $\mathbb{R}$ converging to $x \in \mathbb{R}$. The sequence $(L^{\zeta^T}(.,x_n))_{n\in \mathbb{N}}$ converges
			weakly to $L^{\zeta^T}(.,x)$, that is,
			\begin{equation}
				\lim\limits_{n\rightarrow +\infty}\displaystyle\int_{0}^{T}h(t)\mathrm{d}L^{\zeta^T}(t,x_n)=\displaystyle\int_{0}^{T}h(t)\mathrm{d}L^{\zeta^T}(t,x)
			\end{equation}
			for all bounded and continuous functions $h:[0, T]\longrightarrow \mathbb{R}$.
		\end{enumerate}
	\end{proposition}
	\begin{proof}\begin{enumerate}
			\item[(i)] Since the process $\zeta^T$ is a semi-martingale, according to \cite[Theorem 1.7, Ch. IV]{RY}, there exists a modification of the process $(L^{\zeta^T}(t,x), 0\leq t \leq T, x\in \mathbb{R})$ such that the map $(t,x)\in [0, T]\times\mathbb{R}\longrightarrow L^{\zeta^T}(t,x)$ is continuous in $t$ and cad-lag in $x\in \mathbb{R}$. Moreover, the jump size of $L^{\zeta^T}$ in the $x$ variable is given by 
			\begin{equation}
				L^{\zeta^T}(t,x)-L^{\zeta^T}(t,x-)=2\displaystyle\sum\limits_{i=1}^l\displaystyle\int_0^{t\wedge \nu}\mathbb{I}_{\{\zeta^T_s=x\}}\mathfrak{f}_i(s,\zeta^T_{s})\,\mathrm{d}s\,p_i.
			\end{equation}
			From Theorem \ref{thmsemimartingale} we see that $\langle\zeta^T,\zeta^T\rangle_s=\langle I,I\rangle_s=s\wedge \nu$, hence
			\begin{equation}
				L^{\zeta^T}(t,x)-L^{\zeta^T}(t,x-)=2\displaystyle\sum\limits_{i=1}^l\displaystyle\int_0^{t}\mathbb{I}_{\{\zeta^T_s=x\}} \mathfrak{f}_i(s,\zeta^T_{s})\,\mathrm{d}\langle\zeta^T,\zeta^T\rangle_s\,p_i.
			\end{equation}
			Thus, applying the occupation times formula  to the right-hand side of the previous equality, we see that
			\begin{align}
				L^{\zeta^T}(t,x)-L^{\zeta^T}(t,x-)&=2\displaystyle\sum\limits_{i=1}^l\displaystyle\displaystyle\int_{\{x\}}\displaystyle\int_0^{t} \mathfrak{f}_i(s,y)\mathrm{d}L^{\zeta^T}(s,y)\mathrm{d}y=0.
			\end{align}
			Consequently, the map $(t,x)\in [0, T]\times \mathbb{R}\rightarrow L^{\zeta^T}(t,x)$ is continuous, $\mathbb{P}$-a.s.
			\item[(ii)] It follows from \cite[Corollary 1.9, Ch. VI]{RY} that the local time vanishes outside of the compact
			interval $[-M_t(\omega),M_t(\omega)]$, where 
			\begin{equation}
				M_t(\omega):=\sup\limits_{s\in [0,t]}\vert\zeta^T_s(\omega)\vert,\,t\in [0,\,T],\,\omega\in \Omega.\label{eqM_tbound}
			\end{equation}
			Since the function $x\longrightarrow L^{\zeta^T}(t,x)\,g(x)$  is continuous, it is also bounded.
			\item[(iii)] For all $A\in \mathcal{B}(\mathbb{R}_+)$, we have
			\begin{equation}
				L^{\zeta^T}(A,x)=\displaystyle\int_{A}\mathrm{d}L^{\zeta^T}(t,x),
			\end{equation}
			Let $(x_n)_{n\in \mathbb{N}}$ be a sequence in $\mathbb{R}$ converging to $x \in \mathbb{R}$. The measures $(L^{\zeta^T}(.,x_n))_{n\in \mathbb{N}}$ are
			finite on $\mathbb{R}_+$ and they are supported by $[0, \nu]$. By the continuity of $L^{\zeta^T}(t,.)$, we see that $L^{\zeta^T}(t,x_n)$ converges as $n\rightarrow +\infty$ to $L^{\zeta^T}(t,x)$, for all $t\in [0, T]$, from which it follows that
			\begin{equation}
				\lim\limits_{n\rightarrow +\infty}L^{\zeta^T}([0,t],x_n)=L^{\zeta^T}([0,t],x),\,\,t\in [0,\,T].
			\end{equation}
			We also have 
			\begin{equation}
				L^{\zeta^T}([0, T],x_n)=L^{\zeta^T}([0,\nu],x_n)\overset{n\rightarrow \infty}{\longrightarrow}L^{\zeta^T}([0,\nu],x)=L^{\zeta^T}([0, T],x).
			\end{equation}
			Hence, the measures $(L^{\zeta^T}(.,x_n))_{n\in \mathbb{N}}$ converge weakly to $L^{\zeta^T}(.,x)$.
			See, e.g., \cite[Chapter 3, Section 1]{S}.
		\end{enumerate}
	\end{proof}
	Now we are in position to state our main result of this section.
	\begin{theorem}\label{thmcompensator}
		The compensator of $\nu$ with respect to $\mathbb{F}^{\zeta^T}$ is given by
		\begin{equation}
			K_t=\displaystyle\sum\limits_{k=1}^l\displaystyle\int_{0}^{t\wedge \nu}\dfrac{p_k\dfrac{f(s,z_k)}{p(s,\sigma z_k)}}{\displaystyle\sum\limits_{i=1}^l\bigg[\displaystyle\int_{s}^{T}\Phi_{s,r}(z_k,z_i)f(r,z_i)\mathrm{d}r+F(T,z_i)\Phi_{s,T}(z_k,z_i)\bigg]p_i}\mathrm{d}L^{\zeta^T}(s,\sigma z_k),\label{eqcompensatorexpression}
		\end{equation}
		where,
		\begin{equation}
			\Phi_{s,t}(x,y)=\dfrac{p(t-s,\sigma (y-x))}{p(t,\sigma y)},\,\,s<t,\,\,x,\,y\in\mathbb{R}.
		\end{equation}
	\end{theorem}
	\begin{proof}
		The process $H$ is a bounded, non-negative, increasing, adapted process. It is a sub-martingale and the process $G$ given by 
		\begin{equation}
			G_t:=1-H_t=\mathbb{I}_{\{t<\nu\}},\label{eqG}
		\end{equation}
		is a right-continuous potential of class (D) since
		\begin{equation*}
			\lim\limits_{t\rightarrow T}\mathbb{E}[G_t]=\lim\limits_{t\rightarrow T}\mathbb{P}(\nu >t)=0.
		\end{equation*}
		Let us consider the increasing process $A^h=(A^h_t, 0\leq t \leq T)$ defined by 
		\begin{align}
			A^h_t&=\dfrac{1}{h}\displaystyle\int_0^t\mathbb{P}(s<\nu<s+h \vert \mathcal{F}_s^{\zeta^T})\,\mathrm{d}s.\label{eqA^h}
		\end{align} 
		It follows from \cite[VII.T29]{M} that there exists an integrable, natural, increasing process $K'$, which generates $G$, and this process is unique. For every stopping time $S$, $$K'_S=\lim\limits_{h\rightarrow 0}A^h_S$$ in the sense of the weak topology $\sigma(L^1,L^{\infty})$. From the definition of potential generated by an increasing process, we see that the process given by
		\begin{equation}
			L_t:=G_t+K'_t,\,\,\geq 0,\label{eqK'}
		\end{equation}
		is a martingale. By combining \eqref{eqG} and \eqref{eqK'} we obtain the following decomposition of
		$H$:
		$$ H=1-L+K'. $$
		Therefore, by uniqueness of the decomposition \eqref{eqcompensator}, we can identify the martingale $N$
		with $1 - L$ and we have that
		$K = K'$, up to indistinguishability, which implies that $K'$ is the compensator of $H$. Let us now compute its explicit expression: Let $0< t_0 < t <T$, it follows from \eqref{eqnugivenFt} that, $\mathbb{P}$-a.s.,
		\begin{equation}
			A^h_t-A^h_{t_0}=\displaystyle\sum\limits _{i=1}^{l}\displaystyle\int_{t_0\wedge\nu}^{t\wedge\nu}\dfrac{1}{h}\displaystyle\int_{s}^{s+h}p\bigg(\dfrac{s(r-s)}{r},\zeta^T_{s},\sigma\dfrac{s}{r}z_i\bigg)\,f(r,z_i)\mathrm{d}r\,u(s,\zeta^T_s)\,\mathrm{d}s\,p_i
		\end{equation}
		where,
		\begin{equation}
			u(s,x)=\bigg(\displaystyle\sum\limits _{i=1}^{l}\bigg[\displaystyle\int_{s}^{T}\varphi_{s}^{r,z_i}(x)\,f(r,z_i)\mathrm{d}r+F(T,z_i)\varphi_{s}^{T,z_i}(x)\bigg]\,p_i\bigg)^{-1},\,\,0<s<T,\,\,x\in \mathbb{R}.
		\end{equation}
		Later, we shall verify that for all $i\in\{1,\ldots,l\}$
		\begin{equation}
			\lim\limits_{h\downarrow 0}\displaystyle\int_{t_0\wedge\nu}^{t\wedge\nu}\dfrac{1}{h}\displaystyle\int_{s}^{s+h}p\bigg(\dfrac{s(r-s)}{r},\zeta^T_{s},\dfrac{s}{r}\sigma z_i\bigg)\,[f(r,z_i)-f(s,z_i)]\mathrm{d}ru(s,\zeta^T_s)\,\mathrm{d}s=0.\label{eqlimit0}
		\end{equation}
		Hence, we have to deal with the limit behaviour as $h\downarrow 0$ of
		\begin{multline}
			\displaystyle\sum\limits _{i=1}^{l}\displaystyle\int_{t_0\wedge\nu}^{t\wedge\nu}\dfrac{1}{h}\displaystyle\int_{s}^{s+h}p\bigg(\dfrac{s(r-s)}{r},\zeta^T_{s},\dfrac{s}{r}z_i\bigg)\mathrm{d}rf(s,z_i)u(s,\zeta^T_s)\,\mathrm{d}s\,p_i=\\
			\displaystyle\sum\limits _{i=1}^{l}\displaystyle\int_{t_0\wedge\nu}^{t\wedge\nu}\dfrac{1}{h}\displaystyle\int_{0}^{h}p\bigg(\dfrac{sr}{r+s},\zeta^T_{s},\dfrac{s}{r+s}z_i\bigg)\mathrm{d}rf(s,z_i)u(s,\zeta^T_s)\,\mathrm{d}s\,p_i.\label{eqsum}
		\end{multline}
		For all $i\in \{1,\ldots,n\}$, we have the following estimate
		\begin{small}
			\begin{align}
				&\bigg\vert p\bigg(\dfrac{sr}{r+s},x,\dfrac{s}{r+s}\sigma z_i\bigg)-\exp\bigg(\dfrac{\sigma z_i(\sigma z_i-x)}{s}\bigg)p(r,x,\sigma z_i) \bigg\vert\nonumber\\
				\nonumber\\
				=&\exp\bigg(\dfrac{\sigma z_i(\sigma z_i-x)}{s}\bigg)p(r,x,\sigma z_i)\bigg\vert \sqrt{\dfrac{s+r}{s}}\exp\bigg(-\dfrac{(x-\sigma z_i)^2}{2s}-\dfrac{r\sigma^2z_i^2}{2s(s+r)}\bigg)-1 \bigg\vert\nonumber\\
				\nonumber\\
				\leq& \exp\bigg(\dfrac{\sigma z_i(\sigma z_i-x)}{s}\bigg)p(r,x,\sigma z_i)\bigg[ \sqrt{\dfrac{s+r}{s}}\bigg\vert\exp\bigg(-\dfrac{(x-\sigma z_i)^2}{2s}-\dfrac{r\sigma^2z_i^2}{2s(s+r)}\bigg)-1 \bigg\vert+\bigg\vert \sqrt{\dfrac{s+r}{s}}-1\bigg\vert\bigg]\nonumber\\
				\nonumber\\
				\leq& \exp\bigg(\dfrac{\sigma z_i(\sigma z_i-x)}{s}\bigg)\bigg[\dfrac{1}{\sqrt{2\pi}\vert x-\sigma z_i\vert}\exp\bigg(-\dfrac{1}{2}\bigg)\sqrt{\dfrac{s+1}{s}}\dfrac{(x-\sigma z_i)^2}{2s}+\dfrac{1}{\sqrt{2\pi r}}\bigg(\dfrac{r\sigma^2z_i^2}{2s(s+r)}+\dfrac{r}{2s} \bigg) \bigg]\nonumber\\
				\nonumber\\
				\leq&(c_1\vert x-\sigma z_i \vert+c_2\sqrt{r})C(t_0,t,x),\label{eqestimate}
			\end{align}
		\end{small}
		with some constants $c_1$ and $c_2$, for $0 \leq r \leq  h \leq  T$ and $s \in [t_0, t]$, where
		\begin{equation}
			C(t_0,t,x)=\sup\limits_{1\leq i \leq n}\bigg[\exp\bigg(\dfrac{\sigma z_i(\sigma z_i-x)}{t}\bigg)\vee \exp\bigg(\dfrac{\sigma z_i(\sigma z_i-x)}{t_0}\bigg)\bigg].\label{eqC}
		\end{equation}
		Recall that we have $\mathbb{P}(\tau\geq T)=\displaystyle\sum\limits _{i=1}^{l}F(T,z_i)p_i>0$, then, without loss of generality, we can assume that $F(T,z_1)p_1>0$. Thus, for all $x\in \mathbb{R}$ and $s\in [t_0, t]$, we have
		\begin{align*}
			\displaystyle\sum\limits _{i=1}^{l}\bigg[\displaystyle\int_{s}^{T}\varphi_{s}^{r,z_i}(x)\,f(r,z_i)\mathrm{d}r+F(T,z_i)\varphi_{s}^{T,z_i}(x)\bigg]\,p_i&\geq F(T,z_1)\varphi_{s}^{T,z_1}(x)p_1.
		\end{align*}
		Consequently, for all $x\in \mathbb{R}$, we have
		\begin{equation}
			\sup\limits_{s\in [t_0,t]}u(s,x)\leq \dfrac{\sqrt{2\pi (T-t_0)}}{p_1F(T,z_1)} \exp\bigg(\dfrac{(x-\sigma\,z_1)^2}{2(T-t)}+\dfrac{x^2}{2t_0}\bigg)=D(t_0,t,x).\label{eqestimateg(s,x)}
		\end{equation}
		It is a simple matter to check that the functions $C(t_0,t,x)$ and $D(t_0,t,x)$ defined respectively in \eqref{eqC} and \eqref{eqestimateg(s,x)} are continuous in $x$. Hence, \eqref{eqestimate} and \eqref{eqestimateg(s,x)} show that
		for $0 \leq r \leq h \leq T$ and $s \in [t_0, t]$ 
		\begin{align}
			\displaystyle\sum\limits _{i=1}^{l}\dfrac{1}{h}\displaystyle\int_{0}^{h}\bigg\vert p\bigg(&\dfrac{sr}{r+s},\zeta^T_s,\dfrac{s}{r+s}\sigma z_i\bigg)-\exp\bigg(\dfrac{\sigma z_i(\sigma z_i-\zeta^T_s)}{s}\bigg)p(r,\zeta^T_s,\sigma z_i)\bigg\vert\mathrm{d}ru(s,\zeta^T_s)f(s,z_i)\,p_i\nonumber\\
			&\leq\bigg(c_1\vert \zeta^T_s\vert+c_1\mathbb{E}[\vert Z_T\vert]+c_2\bigg)C(t_0,t,\zeta^T_s)D(t_0,t,\zeta^T_s)\sup\limits_{s\in[t_0,t]}f(s,z_i).\label{eqestimatedominatedconvergence}
		\end{align}
		Note that the right-hand side of \eqref{eqestimatedominatedconvergence} is integrable over $[t_0, t]$ with respect to the Lebesgue measure. On the other hand, using the fact that the function 
		\begin{equation*}
			{\displaystyle r\longrightarrow\begin{cases}
					p(r,x,\sigma z), & \text{if}\,\,r\neq 0,\\
					0, &\text{if}\,\, r=0,
			\end{cases}}
		\end{equation*}
		is continuous if $x\neq \sigma z$, by the fundamental theorem of calculus we have for every $x\neq \sigma z$
		\begin{equation}
			\lim\limits_{h \downarrow 0}\dfrac{1}{h}\displaystyle\int_{0}^{h} p\bigg(\dfrac{sr}{r+s},x,\sigma z\bigg)\mathrm{d}r=0,\,\,\,\text{and}\,\,\,\lim\limits_{h \downarrow 0}\dfrac{1}{h}\displaystyle\int_{0}^{h}p(r,x,\sigma z)\mathrm{d}r=0.
		\end{equation} 
		By the occupation time formula we have
		\begin{equation*}
			\displaystyle\int_{0}^{t\wedge\nu}\displaystyle\sum\limits_{i=1}^{l}\mathbb{I}_{\{\zeta^T_s=\sigma z_i\}}\mathrm{d}s=\displaystyle\sum\limits_{i=1}^{l}\displaystyle\int_{0}^{t}\mathbb{I}_{\{\zeta^T_s=\sigma z_i\}}\mathrm{d}\langle\zeta^T,\zeta^T\rangle_s=\displaystyle\sum\limits_{i=1}^{l}\displaystyle\int_{-\infty}^{+\infty}L^{\zeta^T}(t,x)\mathbb{I}_{\{x=\sigma z_i\}}\mathrm{d}x=0.
		\end{equation*}
		Hence, the set $$\bigcup\limits_{i=1}^l\{0\leq s\leq t\wedge\nu: \zeta^T_s=\sigma z_i\}$$ has Lebesgue measure zero. Consequently, for every $i\in\{1,\ldots,n\}$,
		\begin{small}
			\begin{align}
				&\lim\limits_{h \downarrow 0}\dfrac{1}{h}\displaystyle\int_{0}^{h}\bigg\vert p\bigg(\dfrac{sr}{r+s},x,\dfrac{s}{r+s}\sigma z_i\bigg)-\exp\bigg(\dfrac{\sigma z_i(\sigma z_i-\zeta^T_s)}{s}\bigg)p(r,\zeta^T_s,\sigma z_i)\bigg\vert\mathrm{d}r\nonumber\\
				\leq& \exp\bigg(\dfrac{\sigma z_i(\sigma z_i-\zeta^T_s)}{s}\bigg)\bigg[\lim\limits_{h \downarrow 0} \displaystyle\int_{0}^{h}p\bigg(\dfrac{sr}{r+s},\zeta^T_s,\sigma z_i\bigg)\mathrm{d}r+\lim\limits_{h \downarrow 0}\dfrac{1}{h}\displaystyle\int_{0}^{h}p(r,\zeta^T_s,\sigma z_i)\mathrm{d}r\bigg]=0.\label{eqlimdominatedconvergence}
			\end{align}
		\end{small}
		From \eqref{eqestimatedominatedconvergence}, \eqref{eqlimdominatedconvergence}, and Lebesgue's dominated convergence theorem it follows that for every $i\in\{1,\ldots,l\}$, $\mathbb{P}$-a.s.,
		\begin{align*}
			\displaystyle\int_{t_0\wedge\nu}^{t\wedge\nu}\dfrac{1}{h}\displaystyle\int_{0}^{h}\bigg\vert p\bigg(\dfrac{sr}{r+s},\zeta^T_s,\dfrac{s}{r+s}\sigma z_i\bigg)&-\exp\bigg(\dfrac{\sigma z_i(\sigma z_i-\zeta^T_s)}{s}\bigg)p(r,\zeta^T_s,\sigma z_i)\bigg\vert\mathrm{d}rf(s,z_i)u(s,\zeta^T_s)\,\mathrm{d}s
		\end{align*}
		goes to $0$ as $h \downarrow 0$.
		This means that we have to deal with the limit behaviour as $h\downarrow 0$ of
		\begin{equation*}
			\displaystyle\sum\limits _{i=1}^{l}\displaystyle\int_{t_0\wedge\nu}^{t\wedge\nu}\dfrac{1}{h}\displaystyle\int_{0}^{h}\exp\bigg(\dfrac{\sigma z_i(\sigma z_i-\zeta^T_s)}{s}\bigg)p(r,\zeta^T_s,\sigma z_i)\mathrm{d}rf(s,z_i)u(s,\zeta^T_s)\,\mathrm{d}s\,p_i.
		\end{equation*}
		With the notation,
		\begin{equation}
			(q_i(h,x))_{1\leq i\leq l}:=\bigg(\dfrac{1}{h}\displaystyle\int_{0}^{h}p(r,x,\sigma z_i)\mathrm{d}r\bigg)_{1\leq i\leq l},\,\,0<h\leq T,\,\,x\in \mathbb{R},\label{eqq_h}
		\end{equation}
		the occupation time formula yields that for every $i\in \{1,\ldots,l\}$,
		\begin{align}
			\displaystyle\int_{t_0\wedge\nu}^{t\wedge\nu}\dfrac{1}{h}&\displaystyle\int_{0}^{h}p(r,\zeta^T_s,\sigma z_i)\mathrm{d}r\exp\bigg(\dfrac{\sigma z_i(\sigma z_i-\zeta^T_s)}{s}\bigg)f(s,z_i)u(s,\zeta^T_s)\,\mathrm{d}s\nonumber\\
			&=\displaystyle\int_{t_0}^{t}q_i(h,\zeta^T_s)\exp\bigg(\dfrac{\sigma z_i(\sigma z_i-\zeta^T_s)}{s}\bigg)f(s,z_i)u(s,\zeta^T_s)\,\mathrm{d}\langle\zeta^T,\zeta^T\rangle_s\nonumber\\
			&=\displaystyle\int_{-\infty}^{+\infty}\displaystyle\int_{t_0}^{t}\exp\bigg(\dfrac{\sigma z_i(\sigma z_i-x)}{s}\bigg)f(s,z_i)u(s,x)\mathrm{d}L^{\zeta^T}(s,x)q_i(h,x)\mathrm{d}x,\,\,\mathbb{P}\text{-a.s.}\label{eqoccupation time1}
		\end{align}
		Now we can state the following lemma which will allow to complete the proof of Theorem \ref{thmcompensator}.
		\begin{lemma}\label{lemmahelp}
			The functions
			\begin{equation}
				x\in \mathbb{R}\longrightarrow k_i(x)=\displaystyle\int_{t_0}^{t}\exp\bigg(\dfrac{\sigma z_i(\sigma z_i-x)}{s}\bigg)f(s,z_i)u(s,x)\mathrm{d}L^{\zeta^T}(s,x),\,\,i=1,\ldots,l,\label{eqk}
			\end{equation}
			are continuous and bounded.
		\end{lemma}
		\begin{proof}
			We first need to show the following two statements:
			\begin{enumerate}
				\item[(i)] For all $x\in \mathbb{R}$ and $0 < t_0 < t$, the function $s\in [t_0, t]\longrightarrow u(s,x)$ is continuous.
				\item[(ii)]  For all $0 < t_0 < t$,
				\begin{equation}
					\lim\limits_{n\rightarrow +\infty}\sup\limits_{s\in [t_0,t]}\vert u(s,x_n)-u(s,x)\vert=0,\label{eqg(s,x_n)-g(s,x)}
				\end{equation}
				where $(x_n)_{n\in \mathbb{N}}$ is a sequence converging monotonically to $x\in \mathbb{R}$.
			\end{enumerate}
			Proof of statement (i): we consider the function defined on $[t_0,t]\times\mathbb{R}$ by
			\begin{equation}
				b(s,x)=\displaystyle\sum\limits _{i=1}^{l}b_i(s,x)\,p_i,\label{eqd(s,x)}
			\end{equation}
			where,
			\begin{equation}
				b_i(s,x)=\displaystyle\int_{s}^{T}p\bigg(\dfrac{s(r-s)}{r},x,\dfrac{s}{r}\sigma z_i\bigg)\,f(r,z_i)\mathrm{d}r,\,\,i=1,\ldots,l.
			\end{equation}
			We recall that $$u(s,x)=\bigg[b(s,x)+\displaystyle\sum\limits _{i=1}^{l}F(T,z_i)\varphi_{s}^{T,z_i}(x)p_i\bigg]^{-1},$$ let $s_n$, $s\in[t_0, t]$ such that $s_n \rightarrow s$ as
			$n\rightarrow +\infty$. For every $i\in \{1,\ldots,l\}$, we have 
			\begin{multline}
				b_i(s_n,x)=\displaystyle\int_{t_0}^{T}\mathbb{I}_{\{s_n<r\}}\sqrt{\dfrac{r}{2\pi s_n(r-s_n)}}\exp\bigg(-\dfrac{r
				}{2s_n(r-s_n)}\bigg(x-\dfrac{s_n}{r}\sigma z_i\bigg)^2\bigg)f(r,z_i)\mathrm{d}r\\
				=\displaystyle\int_{t_0}^{t}\mathbb{I}_{\{s_n<r\}}\sqrt{\dfrac{r}{2\pi s_n(r-s_n)}}\exp\bigg(-\dfrac{r
				}{2s_n(r-s_n)}\bigg(x-\dfrac{s_n}{r}\sigma z_i\bigg)^2\bigg)f(r,z_i)\mathrm{d}r\\
				+\displaystyle\int_{t}^{T}\sqrt{\dfrac{r}{2\pi s_n(r-s_n)}}\exp\bigg(-\dfrac{r
				}{2s_n(r-s_n)}\bigg(x-\dfrac{s_n}{r}\sigma z_i\bigg)^2\bigg)f(r,z_i)\mathrm{d}r.\label{eqd(s_n,x)}
			\end{multline}
			For the first integral,
			\begin{multline}
				\displaystyle\int_{t_0}^{t}\mathbb{I}_{(s_n,+\infty)}(r)\sqrt{\dfrac{r}{2\pi s_n(r-s_n)}}\exp\bigg(-\dfrac{r
				}{2s_n(r-s_n)}\bigg(x-\dfrac{s_n}{r}\sigma z_i\bigg)^2\bigg)f(r,z_i)\mathrm{d}r\\
				=\displaystyle\int_{0}^{t}\mathbb{I}_{\{r<t-s_n\}}\sqrt{\dfrac{r+s_n}{2\pi r s_n}}\exp\bigg(-\dfrac{r+s_n
				}{2r s_n}\bigg(x-\dfrac{s_n}{r+s_n}\sigma z_i\bigg)^2\bigg)f(r+s_n,z_i)\mathrm{d}r,
			\end{multline}
			we estimate the integrand of the right hand side by $\sqrt{\frac{t}{\pi \,t_0\, r}}\,\sup\limits_{r\in[0,2t]}f(r,z_i)$, which is integrable over $(0,t]$. Thus, we can apply Lebesgue's theorem to conclude that
			\begin{multline}
				\lim\limits_{n\rightarrow \infty}\displaystyle\int_{t_0}^{t}\mathbb{I}_{(s_n,T]}(r)\sqrt{\dfrac{r}{2\pi s_n(r-s_n)}}\exp\bigg(-\dfrac{r
				}{2s_n(r-s_n)}\bigg(x-\dfrac{s}{r}\sigma z_i\bigg)^2\bigg)f(r,z_i)\mathrm{d}r\\=\displaystyle\int_{t_0}^{t}\mathbb{I}_{(s,T]}(r)\sqrt{\dfrac{r}{2\pi s(r-s)}}\exp\bigg(-\dfrac{r
				}{2s(r-s)}\bigg(x-\dfrac{s}{r}\sigma z_i\bigg)^2\bigg)f(r,z_i)\mathrm{d}r.\label{eqlimits_n1}
			\end{multline}
			Concerning the second integral in \eqref{eqd(s_n,x)}, we estimate the integrand by $ \sqrt{\frac{r}{2\pi t_0(r-t)}}f(r,z_i) $ which is integrable over $[t,T]$. Hence, using Lebesgue's dominated convergence theorem we obtain
			\begin{multline}
				\lim\limits_{n\rightarrow \infty}\displaystyle\int_{t}^{T}\sqrt{\dfrac{r}{2\pi s_n(r-s_n)}}\exp\bigg(-\dfrac{r
				}{2s_n(r-s_n)}\bigg(x-\dfrac{s}{r}\sigma z_i\bigg)^2\bigg)f(r,z_i)\mathrm{d}r\\=\displaystyle\int_{t}^{T}\sqrt{\dfrac{r}{2\pi s(r-s)}}\exp\bigg(-\dfrac{r
				}{2s(r-s)}\bigg(x-\dfrac{s}{r}\sigma z_i\bigg)^2\bigg)f(r,z_i)\mathrm{d}r.\label{eqlimits_n2}
			\end{multline}
			From \eqref{eqlimits_n1} and \eqref{eqlimits_n2}, it follows that 
			\begin{equation*}
				\lim\limits_{n\rightarrow +\infty}b_i(s_n,x)=b_i(s,x),\,\text{for every}\,\,i\in \{1,\ldots,l\}.
			\end{equation*}
		We conclude from \eqref{eqd(s,x)} that the function $b$ is continuous on $[t_0, t]$. Since the functions $s\in [t_0, t]\longrightarrow \varphi_{s}^{T,z_i}(x)$ are also continuous, the function $s\in [t_0, t]\longrightarrow u(s,x)$ is continuous on $[t_0, t]$ for every $x\in \mathbb{R}$.\\
			Proof of statement (ii): using the fact that
			\begin{equation*}
				\vert u(s,x_n)-u(s,x) \vert\leq u(s,x_n)\,u(s,x)\Big[\vert b(s,x_n)-b(s,x) \vert+\displaystyle\sum\limits _{i=1}^{l}F(T,z_i)\vert\varphi_{s}^{T,z_i}(x_n)-\varphi_{s}^{T,z_i}(x)\vert p_i\Big],
			\end{equation*}
			it follows from \eqref{eqestimateg(s,x)} that
			\begin{align}
				\sup\limits_{s\in[t_0,t]}\vert u(s,x_n)-u(s,x) \vert
				\leq& D(t_0,t,x_n)\,D(t_0,t,x)\bigg[\displaystyle\sum\limits_{i=1}^l\,p_i\sup\limits_{s\in[t_0,t]}\vert b_i(s,x_n)-b_i(s,x) \vert\\
				&+\displaystyle\sum\limits _{i=1}^{l}F(T,z_i)p_i\sup\limits_{s\in[t_0,t]}\vert\varphi_{s}^{T,z_i}(x_n)-\varphi_{s}^{T,z_i}(x)\vert\bigg].\label{eqsupg(s,x_n)-g(s,x)}
			\end{align}
			On the other hand, we have
			\begin{align}
				&\vert b_i(s,x_n)-b_i(s,x) \vert=\exp\bigg(\dfrac{z_i^2}{2s}\bigg)\vert \kappa^{i}_2(s,x_n)\kappa^{i}_1(s,x_n)-\kappa^{i}_2(s,x)\kappa^{i}_1(s,x)\vert\nonumber\\
				&\leq \exp\bigg(\dfrac{z_i^2}{2s}\bigg)\Big[\kappa^{i}_2(s,x_n)\vert \kappa^{i}_1(s,x_n)- \kappa^{i}_1(s,x)\vert+ \kappa^{i}_1(s,x)\vert \kappa^{i}_2(s,x_n)-\kappa^{i}_2(s,x)\vert\Big],\label{eqd(s,x_n)-d(s,x)}
			\end{align}
			and
			\begin{align}
				&\vert \varphi_{s}^{T,z_i}(x_n)-\varphi_{s}^{T,z_i}(x) \vert=\dfrac{1}{p(T,\sigma z_i)}\vert p(T-s,\sigma z_i-x_n)p(s,x_n)-p(T-s,\sigma z_i-x)p(s,x)\vert\nonumber\\
				&\leq \sqrt{\dfrac{T}{2\pi s(T-s)}}\exp\bigg(\frac{\sigma^2 z_i^2}{2T}\bigg)\Big[\vert \kappa^{i}_3(s,x_n)- \kappa^{i}_3(s,x)\vert+ \vert \kappa_4(s,x_n)-\kappa_4(s,x)\vert\Big],\label{eqphi(s,x_n)-phi(s,x)}
			\end{align}
			where, for $i\in \{1,\ldots,l\}$,
			\begin{equation}
				\kappa^{i}_1(s,x)=\displaystyle\int_{s}^{+\infty}\sqrt{\dfrac{r}{2\pi s(r-s)}}\exp\bigg(-\dfrac{r
				}{2s(r-s)}(x-\sigma z_i)^2\bigg)\exp\bigg(\dfrac{\sigma^2 z_i^2}{2r}\bigg)f(r,z_i)\mathrm{d}r,
			\end{equation}
			\begin{equation}
				\kappa^{i}_2(s,x)=\exp\bigg(-\dfrac{\sigma z_ix}{s}\bigg),\,\,\,\kappa^{i}_3(s,x)=\exp\bigg(-\dfrac{(x-\sigma z_i)^2}{2(T-s)}\bigg),
			\end{equation}
			and
			\begin{equation}
				\kappa_4(s,x)=\exp\bigg(-\dfrac{x^2}{2s}\bigg),
			\end{equation}
			It is easy to show that, for all $x\in \mathbb{R}$, the functions $s\longrightarrow \kappa^{i}_1(s,x)$ are continuous on $[t_0,t]$. Hence, in \eqref{eqd(s,x_n)-d(s,x)} and \eqref{eqphi(s,x_n)-phi(s,x)} we can pass to the supremum over $[t_0,t]$ and obtain the following 
			\begin{multline}
				\sup\limits_{s\in[t_0,t]}\vert b_i(s,x_n)-b_i(s,x) \vert\\
				\leq \exp\Big(\dfrac{z^2}{2t_0}\Big)\Big[\kappa^{i}_2(t_0,x_n)\vee\kappa^{i}_2(t,x_n)\Big]\sup\limits_{s\in[t_0,t]}\vert \kappa^{i}_1(s,x_n)- \kappa^{i}_1(s,x)\vert\\
				+\exp\Big(\dfrac{z^2}{2t_0}\Big)\sup\limits_{s\in[t_0,t]}\kappa^{i}_1(s,x)\sup\limits_{s\in[t_0,t]}\vert \kappa^{i}_2(s,x_n)-\kappa^{i}_2(s,x)\vert,
			\end{multline}
			and
			\begin{multline}
				\sup\limits_{s\in[t_0,t]}\vert \varphi_{s}^{T,z_i}(x_n)-\varphi_{s}^{T,z_i}(x) \vert
				\leq \sqrt{\dfrac{T}{2\pi t_0(T-t)}}\exp\bigg(\frac{\sigma^2 z_i^2}{2T}\bigg)\\\times\Big[\sup\limits_{s\in[t_0,t]}\vert \kappa^{i}_3(s,x_n)- \kappa^{i}_3(s,x)\vert+ \sup\limits_{s\in[t_0,t]}\vert \kappa_4(s,x_n)-\kappa_4(s,x)\vert\Big],
			\end{multline}
			for every $i\in\{1,\ldots,l\}$. Without loss of generality we can assume that $x_n\leq \sigma z_i$ for all $n\in \mathbb{N}$ or $x_n\geq \sigma z_i$ for all $n\in \mathbb{N}$, that depends on whether $x\geq \sigma z_i$ or $x\leq \sigma z_i$. Since the sequence $x_n$ converges monotonically to $x$, it is easy to see that the sequences of functions $\kappa^{i}_1(.,x_n)$, $\kappa^{i}_2(.,x_n)$, $\kappa^{i}_3(.,x_n)$ and $\kappa_4(.,x_n)$ are monotone and that for all $s\in[t_0,t]$, $\kappa^{i}_1(s,x_n)$, $\kappa^{i}_2(s,x_n)$, $\kappa^{i}_3(s,x_n)$ and $\kappa_4(s,x_n)$ converge to $\kappa^{i}_1(s,x)$, $\kappa^{i}_2(s,x)$, $\kappa^{i}_3(s,x)$ and $\kappa_4(s,x)$, respectively. Furthermore, since the functions $s\longrightarrow \kappa^{i}_1(s,x)$, $s\longrightarrow \kappa^{i}_2(s,x)$, $s\longrightarrow \kappa^{i}_3(s,x)$ and $s\longrightarrow \kappa_4(s,x)$ are also continuous
			on $[t_0, t]$, according to Dini's theorem, $\kappa^{i}_1(.,x_n)$, $\kappa^{i}_2(.,x_n)$, $\kappa^{i}_3(.,x_n)$ and $\kappa_4(.,x_n)$ converge uniformly to $\kappa^{i}_1(.,x)$, $\kappa^{i}_2(.,x)$, $\kappa^{i}_3(.,x)$ and $\kappa_4(.,x)$ on $[t_0, t]$, respectively. This implies that 
			\begin{equation}
				\lim\limits_{n \rightarrow +\infty}\sup\limits_{s\in[t_0,t]}\vert b_i(s,x_n)-b_i(s,x) \vert=0,
			\end{equation}
			and
			\begin{equation}
				\lim\limits_{n \rightarrow +\infty}\sup\limits_{s\in[t_0,t]}\vert \varphi_{s}^{T,z_i}(x_n)-\varphi_{s}^{T,z_i}(x) \vert=0.
			\end{equation}
			Hence, we get \eqref{eqg(s,x_n)-g(s,x)} from \eqref{eqsupg(s,x_n)-g(s,x)}. This completes the proof of statement (ii). We have now all the ingredients to show that the functions $k_i, i=1,\ldots,l$, defined in \eqref{eqk} are bounded and continuous. Let $E$ be a subset of $\mathbb{R}$ such that $E=[-M_t-1,M_t+1]$ where $M$ is defined in \eqref{eqM_tbound}, since for $s\in [0,t]$ and $x\notin E$, $L^{\zeta^T}(t,x)=0$, it is sufficient to show that  $x\longrightarrow k_i(x)$ is continuous on the compact $E$. Let $x_n$ be a sequence from $E$ converging monotonically to $x \in E$, with the notation
			\begin{equation}
				G_i(s,x)=\exp\bigg(\dfrac{\sigma z_i(\sigma z_i-x)}{s}\bigg)u(s,x)f_{\tau}(s),\,\, i=1,\ldots,l,
			\end{equation} 
			we have 
			\begin{multline}
				\vert k_i(x_n)-k_i(x)\vert=\bigg\vert\displaystyle\int_{t_0}^{t}G_i(s,x_n)\mathrm{d}L^{\zeta^T}(s,x_n)-\displaystyle\int_{t_0}^{t}G_i(s,x)\mathrm{d}L^{\zeta^T}(s,x)\bigg\vert
				\\ \leq \displaystyle\int_{t_0}^{t}\vert G_i(s,x_n)-G_i(s,x)\vert\mathrm{d}L^{\zeta^T}(s,x_n)+\bigg\vert\displaystyle\int_{t_0}^{t}G_i(s,x)\mathrm{d}L^{\zeta^T}(s,x)-\displaystyle\int_{t_0}^{t}G_i(s,x)\mathrm{d}L^{\zeta^T}(s,x_n)\bigg\vert\\
				\leq L^{\zeta^T}([t_0,t],x_n)\sup\limits_{s\in [t_0,t]}\vert G_i(s,x_n)-G_i(s,x)\vert \\+\bigg\vert\displaystyle\int_{t_0}^{t}G_i(s,x)\mathrm{d}L^{\zeta^T}(s,x)-\displaystyle\int_{t_0}^{t}G_i(s,x)\mathrm{d}L^{\zeta^T}(s,x_n)\bigg\vert.
			\end{multline}
			It follows from the fact that $\sup\limits_{s\in [t_0,t]}\vert u(s,x_n)-\mathfrak{g}(s,x)\vert$ converges as $n\rightarrow +\infty$ to $0$ that for every $i\in\{1,\ldots,l\}$
			\begin{equation}
				\lim\limits_{n\rightarrow +\infty}\sup\limits_{s\in [t_0,t]}\vert G_i(s,x_n)-G_i(s,x)\vert=0.
			\end{equation}
			On the other hand,  from the the third statement of  Proposition \ref{proplocaltime}, we obtain that 
			\begin{align}
				\lim\limits_{n\rightarrow +\infty}	\bigg\vert\displaystyle\int_{t_0}^{t}G_i(s,x)\mathrm{d}L^{\zeta^T}(s,x)-\displaystyle\int_{t_0}^{t}G_i(s,x)\mathrm{d}L^{\zeta^T}(s,x_n)\bigg\vert=0.
			\end{align}
			Which implies that for every $i\in\{1,\ldots,l\}$,
			\begin{equation}
				\lim\limits_{n\rightarrow +\infty}\vert k_i(x_n)-k_i(x)\vert=0.
			\end{equation}
			Which completes the proof of Lemma \ref{lemmahelp}.
		\end{proof}
		For $i=1,\ldots,l$, let $\mathbb{Q}_h^i$ be the probability measure with
		density $q_i(h,.)$ defined in \eqref{eqq_h}. Observe that $\mathbb{Q}_h^i$ converges weakly as $h\downarrow 0$ to the Dirac measure $\delta_{\sigma z_i}$ at $\sigma z_i$. Thus, for every $i\in\{1,\ldots,l\}$,
		\begin{equation}
			\lim\limits_{h \downarrow 0}	\displaystyle\int_{-\infty}^{+\infty}k_i(x)q_i(h,x)\mathrm{d}x=k_i(\sigma z_i).
		\end{equation}
		Thus, from \eqref{eqoccupation time1} and \eqref{eqk}, we obtain
		\begin{align*}
			\lim\limits_{h \downarrow 0}\displaystyle\int_{t_0\wedge\nu}^{t\wedge\nu}\dfrac{1}{h}\displaystyle\int_{0}^{h}p(r,\zeta^T_s,\sigma z_i)\mathrm{d}r\exp\bigg(\dfrac{\sigma z_i(\sigma z_i-\zeta^T_s)}{s}\bigg)&f(s,z_i)u(s,\zeta^T_s)\,\mathrm{d}s\\&=\displaystyle\int_{t_0}^{t}f(s,z_i)u(s,\sigma z_i)\mathrm{d}L^{\zeta^T}(s,\sigma z_i).
		\end{align*}
		Note that the functions $f(.,z_i)$ is uniformly continuous on $[t_0, t + 1]$. We fix $\varepsilon> 0$ and choose
		$0 < \rho \leq 1$ such that, for every $0 \leq r < \rho$, $\vert f(s+r,z_i)-f(s,z_i)\vert \leq \varepsilon$. For every $i\in \{1,\ldots,l\}$, we have
		\begin{align}
			\limsup\limits_{h\downarrow 0}\bigg\vert &\displaystyle\int_{t_0\wedge\nu}^{t\wedge\nu}\dfrac{1}{h}\displaystyle\int_{0}^{h}p\bigg(\dfrac{sr}{r+s},\zeta^T_{s},\dfrac{s}{r+s}\sigma z_i\bigg)[f(s+r,z_i)-f(s,z_i)]\mathrm{d}r\,u(s,\zeta^T_{s})\mathrm{d}s \bigg\vert\nonumber\\
			\leq& \limsup\limits_{h\downarrow 0}\displaystyle\int_{t_0\wedge \nu}^{t\wedge \nu}\dfrac{1}{h}\displaystyle\int_{0}^{h}p\bigg(\dfrac{sr}{r+s},\zeta^T_{s},\dfrac{s}{r+s}\sigma z_i\bigg)\vert f(s+r,z_i)-f(s,z_i)\vert\mathrm{d}r\,u(s,\zeta^T_{s})\mathrm{d}s \nonumber\\
			=&\varepsilon\, \limsup\limits_{h\downarrow 0}\displaystyle\int_{t_0\wedge \nu}^{t\wedge \nu}\dfrac{1}{h}\displaystyle\int_{0}^{h}p(r,\zeta^T_{s},\sigma z_i)\mathrm{d}r\exp\bigg(\dfrac{\sigma z_i(\sigma z_i-\zeta^T_{s})}{s}\bigg)\,u(s,\zeta^T_{s})\mathrm{d}s\nonumber\\
			=&\varepsilon\, \displaystyle\int_{t_0}^{t}u(s,\sigma z_i)\mathrm{d}L^{\zeta^T}(s,\sigma z_i).
		\end{align}
		Since $\varepsilon > 0$ is chosen arbitrarily and the integral above is, $\mathbb{P}$-a.s., finite, we conclude that \eqref{eqlimit0} holds. Which proves that, $\mathbb{P}$-a.s., $A_t^h-A_{t_0}^h$ converges as $h\downarrow 0$ to
		\begin{align*}
			\displaystyle\sum\limits_{k=1}^l\displaystyle\int_{t_0\wedge \nu}^{t\wedge \nu}\dfrac{p_k\dfrac{f(s,z_k)}{p(s,\sigma z_k)}}{\displaystyle\sum\limits_{i=1}^l\bigg[\displaystyle\int_{s}^{T}\Phi_{s,r}(z_k,z_i)f(r,z_i)\mathrm{d}r+F(T,z_i)\Phi_{s,T}(z_k,z_i)\bigg]p_i}\mathrm{d}L^{\zeta^T}(s,\sigma z_k).
		\end{align*}
		Recalling that $A_t^h-A_{t_0}^h$ converges as $h\downarrow 0$ to $K'_t-K'_{t_0}$ in the sense of the weak topology $\sigma(L^1,L^{\infty})$, using \cite[II.T21]{M} and \cite[II.T23]{M}, we conclude that, $\mathbb{P}$-a.s.,
		\begin{align*}
			K'_t-K'_{t_0}=\displaystyle\sum\limits_{k=1}^l\displaystyle\int_{t_0\wedge \nu}^{t\wedge \nu}\dfrac{p_k\dfrac{f(s,z_k)}{p(s,\sigma z_k)}}{\displaystyle\sum\limits_{i=1}^l\bigg[\displaystyle\int_{s}^{T}\Phi_{s,r}(z_k,z_i)f(r,z_i)\mathrm{d}r+F(T,z_i)\Phi_{s,T}(z_k,z_i)\bigg]p_i}\mathrm{d}L^{\zeta^T}(s,\sigma z_k),
		\end{align*} 
		for all $t_0$, $t$ such that $0<t_0<t$. Passing to the limit as $t_0 \downarrow 0$, we get that, $\mathbb{P}$-a.s.,
		\begin{align}
			K'_t=\displaystyle\sum\limits_{k=1}^l\displaystyle\int_{0}^{t\wedge \nu}\dfrac{p_k\dfrac{f(s,z_k)}{p(s,\sigma z_k)}}{\displaystyle\sum\limits_{i=1}^l\bigg[\displaystyle\int_{s}^{T}\Phi_{s,r}(z_k,z_i)f(r,z_i)\mathrm{d}r+F(T,z_i)\Phi_{s,T}(z_k,z_i)\bigg]p_i}\mathrm{d}L^{\zeta^T}(s,\sigma z_k).\label{eqK'=int}
		\end{align}
		Since the right hand side in \eqref{eqK'=int} is right-continuous in $t$ and that $K=K'$, up to the indistinguishability, the compensator of the random time $\nu$ with respect to $\mathbb{F}^{\zeta^T}$ is the process $K=(K_t, t\geq 0)$ given by
		\begin{equation*}
			K_t=\displaystyle\sum\limits_{k=1}^l\displaystyle\int_{0}^{t\wedge \nu}\dfrac{p_k\dfrac{f(s,z_k)}{p(s,\sigma z_k)}}{\displaystyle\sum\limits_{i=1}^l\bigg[\displaystyle\int_{s}^{T}\Phi_{s,r}(z_k,z_i)f(r,z_i)\mathrm{d}r+F(T,z_i)\Phi_{s,T}(z_k,z_i)\bigg]p_i}\mathrm{d}L^{\zeta^T}(s,\sigma z_k).
		\end{equation*}	
		Which completes the proof.
	\end{proof}
	\begin{corollary}
		The random time $\nu$ is a totally inaccessible stopping time with respect to $\mathbb{F}^{\zeta^T}$. 
	\end{corollary}
	\begin{proof}
		The process $A$ given by \eqref{eqcompensatorexpression} is continuous. By \cite[Corollary 25.18]{Kall} this is a
		necessary and sufficient condition for $\nu$ to be totally inaccessible with respect to $\mathbb{F}^{\zeta^T}$. 
	\end{proof}
	\begin{remark}
		The indicator process $(\mathbb{I}_{\{\nu\leq t\}}, t\geq 0)$ does not admit an intensity with respect to the filtration $\mathbb{F}^{\zeta^T}$ since it is not possible to apply, for instance, Aven's Lemma for computing the compensator (see \cite{A}).
	\end{remark}
	Let $\mathfrak{H}=(\mathfrak{H}_t, t\geq 0)$ be the process defined by 
	\begin{equation}
		\mathfrak{H}_t=Z_TH_t=Z_T\,\mathbb{I}_{\{\nu\leq t\}},\,\,t\geq 0.
	\end{equation}
	We derive the explicit expression of the compensator associated with the process $\mathfrak{H}$ in the following way.
	\begin{proposition}\label{propmathfrak}
		The compensator of $\mathfrak{H}$ with respect to $\mathbb{F}^{\zeta^T}$ is given by
		\begin{equation}
			\mathfrak{K}_t=\displaystyle\sum\limits_{k=1}^l\displaystyle\int_{0}^{t\wedge \nu}\dfrac{p_k\dfrac{f(s,z_k)}{p(s,\sigma z_k)}\zeta^T_s}{\displaystyle\sum\limits_{i=1}^l\bigg[\displaystyle\int_{s}^{T}\Phi_{s,r}(z_k,z_i)f(r,z_i)\mathrm{d}r+F(T,z_i)\Phi_{s,T}(z_k,z_i)\bigg]p_i}\mathrm{d}L^{\zeta^T}(s,\sigma z_k).\label{eqcompensatorH}
		\end{equation}
	\end{proposition}
	\begin{proof}
		Note that $\nu$ is an $\mathcal{F}_{\nu}^{\zeta^T}$-stopping time with compensator $K$ given by \eqref{eqcompensatorexpression}, and $Z_T$ is an $\mathcal{F}_{\nu}^{\zeta^T}$-measurable random variable. Hence, for any non-negative predictable process $U$, we have on one hand
		\begin{align}
			\mathbb{E}\bigg[\displaystyle\int_0^{\infty}U_s\mathrm{d}\mathfrak{H}_s \bigg]=\mathbb{E}[U_{\nu}\,Z_T\,\mathbb{I}_{\{0<\nu<\infty\}}]=\mathbb{E}[U_{\nu}\,Z_T].\label{eq11compensator}
		\end{align}
		On the other hand, since the process $U\,\zeta^T$ is predictable, we have 
		\begin{align}
			\mathbb{E}\bigg[\displaystyle\int_0^{\infty}U_s\mathrm{d}\mathfrak{K}_s \bigg]&=\mathbb{E}\bigg[\displaystyle\int_0^{\infty}U_s\zeta^T_s\mathrm{d}K_s \bigg]\nonumber\\
			&=\mathbb{E}\bigg[\displaystyle\int_0^{\infty}U_s\zeta^T_s\mathrm{d}H_s \bigg]\nonumber\\
			&=\mathbb{E}[U_{\nu}\,\zeta^T_{\nu}\,\mathbb{I}_{\{0<\nu<\infty\}}]\nonumber\\
			&=\mathbb{E}[U_{\nu}\,Z_T].\label{eq22compensator}
		\end{align}
		It follows from \eqref{eq11compensator} and \eqref{eq22compensator} that for any non-negative predictable process $U$, we have
		\begin{align*}
			\mathbb{E}\bigg[\displaystyle\int_0^{\infty}U_s\mathrm{d}\mathfrak{H}_s \bigg]=\mathbb{E}\bigg[\displaystyle\int_0^{\infty}U_s\mathrm{d}\mathfrak{K}_s \bigg].
		\end{align*}
		Therefore, the process $\mathfrak{K}$ defined in \eqref{eqcompensatorH} is the compensator of $\mathfrak{H}$ with respect to $\mathbb{F}^{\zeta^T}$. Which is the desired result. 
	\end{proof}
	\section*{Conclusion}
	\quad\,\,This paper delves into the intricate task of modelling information flows within financial markets, by extending the information-based asset pricing framework of Brody, Hughston \& Macrina. The primary improvement lies in the integration of a default scenario for the underlying issuer. In this augmentation, the framework extends its reach to encompass a noisy information process that undergoes evolution over a potentially random time-horizon. The duration of the information flow is contingent upon whether the default occurs or not before a fixed future time when the value of the cash-flow would be revealed.\\
	In detail, the information flow pertaining to a non-defaultable  cash flow $Z_T$, scheduled for realization at time $T$ and the potential bankruptcy time $\tau$ of the asset’s writer are modelled through the natural completed filtration generated by a Brownian random bridge $\zeta^T$ with length $\nu=\tau \wedge T$ and pinning point $\sigma Z_T$, where $\sigma$ is a constant. Naturally, the time of bankruptcy may depend on the cash flow, hence, independence between $Z_T$ and $\tau$ is not assumed.
	
	We utilize the Markov property of the market information process $\zeta^T$ and the stopping time property of $\nu=\tau \wedge T$ to derive the pricing formula for a special option. Additionally, we employ the canonical decomposition of $\zeta^T$, i.e., the Doob-Meyer decomposition as semimartingales in its own filtration, to derive a stochastic differential equation satisfied by the price process.
	
	Adopting the methodology pioneered by P.-A.~Meyer, we explicitly compute the compensator of $\nu=\tau \wedge T$. Furthermore, through a well-established equivalence between the categories of stopping times and the regularity of their compensators, we demonstrate that $\nu=\tau \wedge T$ is totally inaccessible. This characteristic proves particularly relevant in financial markets where predicting the writer's default time is impossible until the actual default event unfolds, emphasizing the unpredictability inherent in such systems. 
	\vspace{0,5cm}
	
	\textbf{Acknowledgments:} I would like to express my deep gratitude to the referee for carefully reading the manuscript and providing invaluable comments.
	
	\vspace{0,2cm}


\begin{thebibliography}{99}
		\bibitem{AK}
		Alili, L.; Kyprianou, A. E. Some remarks on first passage of L\'evy processes, the American put and pasting principles. Ann. Appl. Probab. 15 (2005), No. 3, 2062--2080.
		\bibitem{APP}
		Alili, L.; Patie, P.; Pedersen, J. L. Representations of the first hitting time density of an Ornstein-Uhlenbeck process. Stoch. Models 21 (2005), No. 4, 967--980.
		\bibitem{AJ}
		Aksamit, A.; Jeanblanc, M. \textit{Enlargements of Filtrations with
			Finance in view}. Springer, (2017).
		\bibitem{A}
		Aven, T. A theorem for determining the compensator of a counting process. Scand. J. Statist. 12 (1985), No. 1, 69--72.
		\bibitem{BBE}
		Bedini, M. L.; Buckdahn, R.; Engelbert, H. J. Brownian bridges on random intervals. Theory Probab. Appl. 61 (2017), No. 1, 15--39.
		\bibitem{BBEcompensator}
		Bedini, M. L.; Buckdahn, R.; Engelbert, H. J. On the compensator of the default
		process in an information-based model. Probab. Uncertain. Quant. Risk 2 (2017), No. 10, 1--21.
		\bibitem{BS}
		Borodin, A.; Salminen, P. \textit{Handbook of Brownian Motion: Facts and Formulae}. Birkh\"auser, second edition, 2002.
		\bibitem{BHM2007}
		Brody, D. C.; Hughston, L. P.; Macrina, A. Beyond hazard rates: a new framework to credit-risk modelling. In Advances in mathematical finance (eds M. Fu, R. Jarrow, J.-Y. J. Yen and R. Elliott), pp. 231–257, (2007).
		\bibitem{BHM2010}
		Brody, D. C.; Hughston, L. P.; Macrina, A. Credit risk, market sentiment and randomly-timed default. Stochastic analysis (2010), 267--280, Springer, Heidelberg, 2011.
		\bibitem{BHM2008}
		Brody, D. C.; Hughston, L. P.; Macrina, A. Information-based asset pricing. Int. J. Theor. Appl. Finance 11 (2008), No. 1, 107--142.
		\bibitem{BHM}
		Brody, D.C.; Hughston, L.P.; and  Macrina, A. Dam rain and cumulative gain. Proc. R. Soc. Lond. Ser. A Math. Phys. Eng. Sci. 464 (2008), No. 2095, 1801--1822.
		\bibitem{BL}
		Brody, D. C.; Law, Y. T. Pricing of defaultable bonds with random information flow. Appl. Math. Finance 22 (2015), No. 5, 399--420.
		\bibitem{EHL}
		Erraoui, M.; Hilbert, A.; Louriki, M. Bridges with random length: gamma case. J. Theoret. Probab. 33 (2020), No. 2, 931--953.
		\bibitem{EHL(Levy)}
		Erraoui, M.; Hilbert, A.; Louriki, On a L\'evy process pinned at random time. Forum Math. 33 (2021), No. 2, 397--417.
		\bibitem{EHL(Brownian-Levy)}
		Erraoui, M.; Hilbert, A.; Louriki, M. On an extension of the Brownian Bridge with some applications in finance. (2021), arXiv:2110.01316v1.
		\bibitem{EL}
		Erraoui, M.; Louriki, M. Bridges with random length: Gaussian-Markovian case. Markov Process. Related Fields 24 (2018), No. 4, 669--693.
		\bibitem{JMP}
		Janson, S.; M'Baye, S.; Protter, P. Absolutely continuous compensators. Int. J. Theor. Appl. Finance 14 (2011), No. 3, 335--351.
		\bibitem{JYC}
		Jeanblanc, M.; Yor, M.; Chesney, M. \textit{Mathematical Methods for Financial Markets}. Springer, First edition, (2009).
		\bibitem{HHM}
		Hoyle, E.;  Hughston, L.P.;  Macrina, A. L\'evy random bridges and the modelling of financial information. Stochastic Process. Appl. 121 (2011), No. 4, 856--884.
		\bibitem{HMM}
		Hoyle, E.; Macrina, A.; Meng\"ut\"urk, L. A. Modulated information flows in financial markets. Int. J. Theor. Appl. Finance 23 (2020), No. 4, 2050026, 35 pp.
		\bibitem{Kall}
		Kallenberg, O. \textit{Foundations of Modern Probability}. Springer-Verlag, New-York, Second edition, (2002).
		\bibitem{L}
		Louriki, M. Brownian bridge with random length and pinning point for modelling of financial information. (2022), Stochastics 94 (2022), No. 7, 973--1002.
		\bibitem{Men}
		Meng\"ut\"urk, L.A., From Irrevocably Modulated Filtrations to Dynamical Equations Over Random Networks,  J. Theoret. Probab. 36 (2023), 845--875.
		\bibitem{Merton}
		Merton, R. C. On the pricing of corporate debt: the risk structure of interest rates. J. Finance 29 (1974), No. 2, 449--470.
		\bibitem{M}
		Meyer, P.A. \textit{Probability and Potentials}, Ginn (Blaisdell), Waltham, Massachusetts, (1966).
		\bibitem{RuYu}
		Rutkowski, M; Yu, N. An extension of the Brody-Hughston-Macrina approach to modelling of defaultable bonds. Int. J. Theor. Appl. Finance. 10 (2007), 557--589.
		\bibitem{N}
		Nikeghbali, A. An essay on the general theory of stochastic processes. Probab. Surv. 3 (2006), 345--412.
		\bibitem{P}
		Protter, P. \textit{Stochastic Integration and Differential Equations}. 2nd edn. Springer, Berlin, (2005).
		\bibitem{RY}
		Revuz, D.; Yor, M. \textit{Continuous Martingales and Brownian Motion}. Springer-Verlag, Berlin, Third edition, (1999).
		\bibitem{S}
		Shiryaev, A. N. \textit{Probability. 1}.  Springer, New York, Third edition, (2016). 
	\end{thebibliography}
\end{document}